\documentclass[11pt,letterpaper,reqno]{amsart}
\usepackage{url}
\usepackage{amsthm}
\usepackage{amsmath}
\usepackage{amssymb}
\usepackage{amsfonts}

\usepackage{enumitem}
\usepackage{mathscinet}
\usepackage{lipsum}
\usepackage[english]{babel}
\usepackage[autostyle]{csquotes}
\usepackage{bbold}

\usepackage[english]{babel}

\usepackage[colorlinks=true,
    linkcolor=red,
    citecolor=blue]{hyperref}

\usepackage{graphicx}

\usepackage[utf8]{inputenc}
\usepackage[english]{babel}

\usepackage{color}
\usepackage{comment}

\newtheorem{thm}{Theorem}[section]
\newtheorem{definition}[thm]{Definition}
\newtheorem{theorem}[thm]{Theorem}
\newtheorem{corollary}[thm]{Corollary}
\newtheorem*{oseledec*}{Oseledec Theorem}
\newtheorem{cor}[thm]{Corollary}

\newtheorem{lemma}[thm]{Lemma}
\newtheorem{prop}[thm]{Proposition}

\newtheorem{remark}[thm]{Remark}
\theoremstyle{definition}

\makeatletter
\def\moverlay{\mathpalette\mov@rlay}
\def\mov@rlay#1#2{\leavevmode\vtop{%
   \baselineskip\z@skip \lineskiplimit-\maxdimen
   \ialign{\hfil$\m@th#1##$\hfil\cr#2\crcr}}}
\newcommand{\charfusion}[3][\mathord]{
    #1{\ifx#1\mathop\vphantom{#2}\fi
        \mathpalette\mov@rlay{#2\cr#3}
      }
    \ifx#1\mathop\expandafter\displaylimits\fi}
\makeatother

\let\wh\widehat
\let\wt\widetilde

\newcommand{\nocontentsline}[3]{}
\newcommand{\tocless}[2]{\bgroup\let\addcontentsline=\nocontentsline#1{#2}\egroup}

\usepackage{wasysym}

\def\Jac{\ensuremath{\mathrm{Jac}}}

\def\Hh{\ensuremath{\mathfrak{H}}}

\def\L{\ensuremath{\widehat{\mathcal{L}}}}

\def\Uu{\ensuremath{{\mathcal{U}}}}

\title[Local Limit Theorem and Statistical Properties]{Thermodynamic Formalism Out of Equilibrium Part III: Local Limit Theorem and Statistical Properties}

\begin{document}

\author{S. Ben Ovadia, Y. Hafouta}
\date{}

\begin{abstract}
We study limit theorems for random dynamical systems. We recast the random dynamics via the skew-product of the two-point motion (the Varadhan trick). Using this approach we prove several statistical properties for the system, including quenched central limit theorems with rates, quenched local limit theorem, large deviations, and decay of correlations. Then, in applications we show how to verify the general conditions using spectral properties of appropriate operators. A key application is the study of effectively expanding on average random diffeomorphisms (in any dimension, which are allowed to be dissipative). In particular, this is the first instance of proving local limit theorems in such a broad setting, without imposing restrictions on the dynamics. We provide several new examples.
\begin{comment}
Given a random dynamical system, the two point motion is the skew product generated  by the cartesian square of the  system.
 We prove effective quenched decay of correlations, effective exponential concentration inequalities, central limit theorems with rates and local limit theorems relying on properties of the skew products and the two point motion, reducing the random dynamical system to a deterministic one. All of our results are new already for iid maps, but we provide several non iid examples. Our methods neither  pass through a spectral gap for the random transfer operators, nor through non-invertible dynamics. 
%The $N$-point motion is deterministic dynamical system induced by the original system. We show that exponential mixing for the $N$-point motion for every $N$ implies a plethora of statistical properties for the random system, including the local limit theorem. In particular, our methods do not pass through a spectral gap, nor through non-invertible dynamics.
\end{comment}
\end{abstract}

\newcommand{\Addresses}{{% additional braces for segregating \footnotesize
  \bigskip
  \footnotesize

  S.~Ben Ovadia, \textsc{Einstein Institute of Mathematics, The Hebrew University of Jerusalem, 91904 Jerusalem, Israel}. \\ \textit{E-mail address}: \texttt{Snir.BenOvadia@mail.huji.ac.il}
  
  Y.~Hafouta, \textsc{Department of Mathematics, Ben-Gurion University of the Negev, 8410501 Be'er Sheva, Israel}. \\ \textit{E-mail address}: \texttt{yeor.hafouta@mail.huji.ac.il}
}}

\maketitle

\tableofcontents

\section{Introduction}
Let $(\Sigma,\mathcal F,\mu, T)$ be a probability preserving system. Let $X$ be a measurable space and let $f_\omega:X\to X,\, \omega\in\Sigma$ be measurable maps. A random dynamical system is  formed by fixing  $\omega$ and considering orbits of  point $x$ given by 
$$
x,\, f_\omega x,\, f_{T\omega}\circ f_\omega x,...
$$
This means that the local laws of physics are governed by the orbits $(T^j\omega)_{j\geq0}$.
The system $(\Sigma,\mathcal F,\mu, T)$ is very often referred to as the driving system, and it is interpreted as a random noise or an external force. For instance, one can consider $(\Sigma,\mathcal F,\mu, T)$ as a bigger dynamical system like the sun (vs planet earth) or like the moon when analyzing the motion of the ocean.

Ergodic theory of random dynamical systems has attracted a lot of attention in the past decades, see \cite{Arnold98, Cong97, Crauel2002, Kifer86, LiuQian95, KiferLiu}.
We refer to  the introduction of \cite[Chapter 5]{KiferLiu} for a historical discussion and applications to, for instance, statistical physics, economy and meteorology etc.

Our main interest in the paper is quencehd limit theorems for random dynamical systems.
Let $A:X\to\mathbb R$ be a sufficiently regular function. For a fixed $\omega$ we consider random Birkhoff sums 
\begin{equation}\label{SumsIntro}
S_n^\omega A=\sum_{j=0}^{n-1}A\circ f_{T^{j-1}\omega}\circ\cdots\circ f_{T\omega}\circ f_{\omega}.    
\end{equation} 
Quenched limit theorems are limit theorems for the sequences $S_n^\omega A$, considered as random variables with respect to an appropriate measure $\widehat{\mu}_\omega$ on $X$, where $\omega$ is fixed but comes from a set of probability one.

The literature on statistical properties (i.e. limit theorems) of random dynamical systems  exploded in recent years. In \cite{Cogburn} central limit theorems were studied for Markov chains in random dynamical environment.
In \cite{Kifer1998} central limit theorems were studied for a variety of random dynamical systems, while in \cite{Kifer1996} large deviations were obtained.
In \cite{DavorCMP, DavorTAMS} quenched central limit theorems and large deviations were obtained. Quenched (standardized) Berry-Esseen theorems were obtained in \cite{DH1, HK, YHYT}. The quenched local CLT was studied in  \cite{HK, DavorCMP, DavorTAMS, YHYT}. We stress that most of these results apply to either random expanding transformation $f_\omega$ or small random perturbations of Anosov maps. The main idea in the proofs in the results above is based on properties of the random transfer operators (namely on random Perron-Frobenius theorems). 
See also \cite{Bk95} for results for certain classes of uniformly hyperbolic maps.

Exponential decay of correlations for random hyperbolic dynamical systems has been also studied extensively in literature. Uniform exponential decay of correlations for uniformly expanding maps by now is a classical result. Using symbolic representations the same holds for small random perturbations of Axiom A maps. 
Let us also mention  \cite{Buzzi} for random non-uniformly expanding maps and \cite{ABR} for some classes of random partially hyperbolic dynamical systems. In \cite{Buzzi} there is no regularity on the constant appearing the decay of correlations, and we refer to \cite{Haf1, Haf2} for effective quenched decay of correlations (namely, where these random constants satisfy some integrability conditions).  However, like in the case of the central limit theorem, all of the current results assume some kind of hperbolicity.

Our approach in this manuscript is different. Instead of imposing restrictions like hyperbolicity we use an approach initiated by Varadhan  \cite{Var1,Var2}, and reduce several limit theorems to verifying classical conditions on the level of the skew product 
$$
F_1(\omega,x)=(T\omega,f_\omega x)
$$ 
and the two point motion 
$$
F_2(\omega,x,y)=(T\omega, f_\omega x,f_\omega y)
$$ 
which is the skew product of $f_\omega\times f_\omega$. This approach resembles the one taken in  \cite{ALS,ANV,DimaJonEquivalence, DKK04}, 
when $(\Sigma,\mathcal F,\mu, T)$ is a Bernoulli shift and $f_{\omega}$ depends only on the $0$-th coordinate. However, our approach does not require the system $(\Sigma,\mathcal F,\mu, T)$ to be  a Bernoulli shift, and even in that case most of our results are new. In particular the local central limit theorem is new.

Our results include central limit theorems for $S_n^\omega A$ with rates, local central limit theorems, effective exponential decay of correlations and effective exponential concentration inequality. We then will verify our general assumptions for five classes of examples.  The first one is iid effectively expanding on average deffeomorphims (which include conservative co-expanding on average iid random diffeomorphims as in \cite{JonDima}). These form a $C^1$-open class random diffeomorphisms, which are allowed to be (very) dissipative. We prove our limits theorems under a decay of correlations assumption for the (not necessarily invariant) volume measure. The second one is iid non-invertible expanding interval maps (c.f. \cite{ANV}).
The third are random hyperbolic maps, with $(\Sigma,\mathcal F,\mu,T)$ being a hyperbolic system (the main novelty is that the maps are not close to a single map). The fourth is random contracting on average diffeomorphisms, with $(\Sigma,\mathcal F,\mu,T)$ being a subshift of finite type (c.f. \cite{Mostly} in the iid case). The last example are Markov chains in random Markovian environment (i.e. $(\Sigma,\mathcal F,\mu,T)$ is a Markov shift).

\section{Preliminaries and main results}\label{Prem}
Let $(\Sigma,\mathcal F, \mu)$ be a probability space and let $T:\Sigma\to\Sigma$ be a measurable probability preserving transformation. Let us consider a measurable space $X$ and let $X_\omega,\omega\in\Sigma$ be random subsets of $X$.  We  assume that the set 
$$
\mathcal E=\{(\omega,x): \omega\in\Sigma, x\in X_\omega\}\subset \Sigma\times X
$$
is a measurable subset of $\Sigma\times X$.
For instance we can consider the case when $X_\omega=X$ or when $X$ is a metric space and $X_\omega$ are random closed subsets. 
 
Let $f_\omega:X_\omega\to X_{T\omega}$ be maps such that the skew product $F_1:\mathcal E\to\mathcal E$ given by 
$$
F_1(\omega,x)=(T\omega, f_\omega x)
$$
is measurable. Denote 
$$
\mathcal E_2=\{(\omega,x,y): \omega\in\Sigma,\, x,y\in\mathcal E_\omega\}\subset \Sigma\times X\times X.
$$
Recall that the two point motion is the map $F_2:\mathcal E_2\to \mathcal E_2$ given by 
$$
F_2(\omega,x,y)=(T\omega, f_\omega x, f_\omega y).
$$
This is the skew product corresponding to the maps $f_\omega\times f_\omega$.
Henceforth, we fix some probability measure $\wh \nu$ on $\mathcal E$ 
and write
$$
\wh \mu=\int \widehat{\mu}_\omega d\mu(\omega)
$$
where $\widehat{\mu}_\omega$ is a probability measure on $X_\omega$.
Note that when $\wh \mu$ is
invariant under $F_1$ then $(f_\omega)_*\widehat{\mu}_\omega=\mu_{T\omega}$ for $\mu$-a.a. $\omega$. Denote $\wh \mu^{(2)}=\int \wh{\mu}_{\omega}\times\widehat{\mu}_\omega\,d\mu(\omega)$. Note that when $\widehat{\mu}$ is $F_1$-invariant then $\wh \mu^{(2)}$ is an $F_2$ invariant probability measure.

 Given a measurable function $A:X\to\mathbb R$ our goal is to prove limit theorems for partial sums of the form 
$$
S_n^\omega A=\sum_{j=0}^{n-1}A\circ f_{T^{j-1}\omega}\circ\cdots\circ f_{T\omega}\circ f_{\omega}
$$
considered as random variables on the probability space $(X_\omega,\widehat{\mu}_\omega)$, where $\omega$ belongs to a measurable set $\Gamma$ such that $\mu(\Gamma)=1$. It is also convenient to lift $A$ to $\mathcal E$ and write $A(\omega,x)=A(x)$ and 
$$
S_n A=\sum_{j=0}^{n-1}A\circ F_1^j.
$$
Then 
$$
S_n^\omega(\cdot)=S_n A(\omega,\cdot).
$$
Denote $\wh A(x,y)=A(x)-A(y)$, which is a function on $X\times X$, and lift it to a function on 
$\mathcal E_2$ by setting $\wh A(\omega,x,y)=\wh A(x,y)$. Denote also 
$$
S_n \wh A=\sum_{j=0}^{n-1}\wh A\circ F_2^j.
$$
Then, our approach in proving statistical properties of $S_n^\omega A$ is to reduce them the certain standard assumptions on the characteristic functions of $S_n A$ and $S_n \wh A$.

\subsection{Central limit theorem (CLT) with rates}\label{CLT}
\begin{definition}[Quenched CLT]
We say that $S_n^\omega A$ obeys the quenched CLT (with deterministic centering) if there is $b>0$ such that for $\mu$-a.a. $\omega$ for all $t\in\mathbb R$,
$$
\lim_{n\to\infty}\widehat{\mu}_\omega(\{x: (bn)^{-1/2}S_n^\omega A(x)\leq t\})=\Phi(t):=\frac{1}{\sqrt{2\pi}}\int_{-\infty}^t e^{-x^2/2}dx.
$$
\end{definition}
Next, let 
$$
\varphi_n(t)=\int e^{it S_n\wh A/\sqrt n}d\widehat{\mu}^{(2)}, \,\,\phi_n(t)=\int e^{it S_nA/\sqrt n}d\widehat{\mu}
$$
and 
$$
\varphi_{\omega,n}(t)=\int e^{it S_n^\omega  A/\sqrt n}d\widehat{\mu}_\omega.
$$

\begin{theorem}[Quenched CLT with rates]\label{BE}
 Suppose that there exist $\eta>\zeta>0$, $c>0$ and $0\leq a\leq 1$ such that 
 \begin{equation}\label{CF1}
 \int_{0\leq |t|\leq T_n}\frac{|\varphi_n(t)-e^{-t^2}|}{|t|^a}dt=O(n^{-\eta})    
 \end{equation}
 and 
 \begin{equation}\label{CF2}
\int_{0\leq |t|\leq T_n}\frac{|\phi_n(t)-e^{-t^2/2}|}{|t|^a}dt=O(n^{-\eta})
 \end{equation}
 for some sequence increasing sequence $T_n>0$ such that $T_n=O(n^{\zeta})$. Take $\alpha>\frac{1}{\eta-\zeta}$ and then take $r$ such that $\alpha(\eta-\zeta)>1+r$. Set $\epsilon=\alpha(\eta-\zeta)-1-r$. Assume also that there exist $\delta_0<1/\alpha$ such that  $\|S_n^\omega A-S_k^\omega\|_{L^1}=O(n^{\delta_0/2}(n-k)^{1/2})$, for all $k<n$, $\mu$-a.a. $\omega$.
 
 Then for every $0<\delta<\epsilon/2$  for $\mu-$almost all $\omega$ we have 
 $$
\int_{|t|\leq T_n}\frac{|\varphi_{\omega,n}(t)-e^{-t^2/2}|}{|t|^a}dt=O(n^{-(2-a)r\alpha+\delta_0/2})+
O(n^{-\frac{1}{2\alpha}+(2-a)\zeta})
$$
$$
+O(n^{-\delta/\alpha})+O(n^{-1/(2\alpha)+\delta_0/2}):=O(n^{-\beta}).
 $$
 In particular, if $a=1$ then 
 $$
\sup_{t\in\mathbb R}|\widehat{\mu}_\omega\{S_n^\omega A\leq t\sqrt n\}-\Phi(t)|=O(n^{-\beta}+T_n^{-1}).
 $$
\end{theorem}
\begin{remark}\label{Rem1}
In applications, conditions \eqref{CF1} and \eqref{CF2} can be verified through spectral gap assumptions of complex perturbations of appropriate operators. Yet we decided to state Theorem \ref{BE} under these weaker conditions.   
\end{remark}

To see when the growth conditions on $\|S_{n}^\omega A-S_{k}^\omega A\|_{L^1(\wh\mu_\omega)}$ are satisfied we have the following result.

\begin{prop}\label{CondProp}
Suppose either that the measure $\wh\mu$ is $F_1$-invariant or that $\widehat{\mu}_\omega$ and $X_\omega$ do not depend on  $\omega$. In both cases assume that $\nu(A)=0$ where $\nu$ is the projection of $\wh\mu$ on $X$. 
\vskip0.1cm
%(i) Let $k\in\mathbb N$ and let 
%$$
%F_{2k}(\omega,x_1,...,x_{2k})=(T\omega, f_{\omega}x_1,...,f_{\omega}x_{2k})
%$$
%be the $2k$-point motion. Let $\wh\mu^{(2k)}_\omega=\mu_\omega^{\otimes 2k}$ and $\wh\mu^{(2k)}=\int\wh\mu^{(2k)}_\omega\, d\mu(\omega)$. Denote $A_i(x_1,...,x_{2k})=A(x_i)$.
%Then for all $\delta>0$ we have 
%$$
%\wh{\mu}_\omega(S_n^\omega A)=O(n^{1/2+1/(2k)+\delta}), \mu-a.s.
%$$
%if there exists $\rho\in(0,1), C>0$ such that for all $j_1\leq j_2\leq \ldots \leq j_{2k}$ with at least one $j_i$ appearing only once 
%we have
%$$
%\left|\mathbb E_{\wh\mu^{(2k)}}\left[\prod_{i=1}^{2k}A_i\circ F_{2k}^{j_i}\right]\right|\leq C\rho^{L}
%$$
%where $L$ is the maximum of the distance of $j_i$ to the set $\{j_\ell: \ell\not=i\}$ for $i$'s such that $j_i$ appears only once.
If $A$ is bounded then $\|S_n^\omega A-S_k^\omega A\|_{L^2(\wh\mu_\omega)}=O(n^{\delta}(n-k)^{1/2})$  ($\mu$-almost surely) for all $\delta>0$ if 
$$
|\mathrm{Cov}_{\wh\mu^{(2)}}(\widetilde A,\widetilde A\circ F_2^n)|
$$
decays exponentially fast in $n$, where $\widetilde A(x,y)=A(x)A(y)$.
\end{prop}

%{\color{red}
%\begin{corollary}[Spectral gap CLT]
%Suppose that there exist operator measure...and operator $\mathcal L_t$...    
%\end{corollary}
%}

In \cite{DimaJonEquivalence} the CLT was proven using properties of the two point motion for iid map $(f_{T^j\omega})_{j=0}^\infty$. In the above theorem we will show that a similar approach yields CLT rates without the iid assumption. The rates are far from optimal, but will be heavily used in the proof of the local CLT below (for that we apply the results with $T_n=\ln n$ and $a=0$ and so we can take $\zeta$ arbitrarily small).

Let us note that as opposed to the case of random small perturbations of Anosov maps \cite{DavorTAMS} and expanding maps $f_\omega$ \cite{HK}, the CLT holds without fiberwise centering, and instead in applications we will only need that $\int A\,d\widehat{\mu}=\int\widehat{\mu}_\omega(A)d\mu(\omega)=0$. Also, we are able to get rates with the usual $\sqrt{n}$ normalization instead of the random normalization by $\sqrt{\mathrm{Var}_{\widehat{\mu}_\omega}(S_n^\omega A)}$. The price to pay is that the rates are suboptimal.

Finally, when $a=1$ the theorem becomes effective when $\zeta$ is small. Note also that when the skew product and the two point motion have spectral gap (or the annealed operators in the iid case) then we can take $a=1$ and $\eta=1/2$ as long as $\zeta<\frac12$.

\subsection{The local CLT (LLT)}\label{LCLT}
The CLT concerns the asymptotic behavior of probabilities of the form $\widehat{\mu}_\omega(S_n^\omega A\in I_n)$ for intervals $I_n$ of size of order $\sqrt n$. The LLT concerns the asymptotic behavior  of probabilities of the form $\widehat{\mu}_\omega(S_n^\omega A\in I)$ for intervals $I$ of unit size.
The LLT  has origins in the De Moivre-Laplace theorem. In fact, the early calculation of De Moivre (1738) seem to be the fist work in which the CLT appears, so in a sense the LLT precedes the CLT itself. The LLT has applications to various areas of mathematics including mathematical physics
\cite{BLRB, DN16, DN-Mech, DT77, EM22, Kh49, LPRS, PS23}, number theory \cite{ADDS, Beck10, DS17}, 
geometry (\cite{LL15}), PDEs \cite{HKN}, and combinatorics (\cite{Can76,BR83, GR92, Hw98}).
Applications to dynamics include  abelian covers (\cite{BFRT, CaRa, DNP, OP19}),
suspensions flows (\cite{DN-llt}), skew products (\cite{Br05, DSL16, DDKN1, DDKN2, LB06}), and homogeneous dynamics (\cite{BeQu, BRLLT05, BRLLT23, Guiv2015, Hough2019}).

Consider partial sums $S_n^\omega A$ of the form \eqref{SumsIntro}. Let $\widehat{\mu}_\omega$ be a random measure on $X$ and let $\widehat{\mu}=\int \widehat{\mu}_\omega d\mu(\omega)$ and $\widehat{\mu}^{(2)}=\int \widehat{\mu}_\omega\times\widehat{\mu}_\omega d\mu(\omega)$.

\begin{definition}[Quenched LLT]
We say  that $S_n^\omega A$ obeys the  LLT (with deterministic centreing) with local Radon measure $\nu$, arithmeticity set $R$ and variance $a>0$ if $\mu$-almost surely
for every continuous function $G:\mathbb R\to \mathbb R$ with compact support or an indicator of a bounded interval we have
$$
\lim_{n\to\infty}\sup_{u\in R}\left|\sqrt{2a\pi n}\int G(S_n^\omega A-u)d\widehat{\mu}_\omega-e^{-\frac{u^2}{2an}}\int G\,d\nu\right|=0.
$$
\end{definition}
In general, the optimal CLT uniform rate is $O(n^{-1/2})$. For a fixed interval  length, the LLT provides the right correction term to the optimal CLT rates, with error term of size $o(n^{-1/2})$.

%{\color{red}
%\begin{corollary}[Spectral gap non-lattice LLT]
%Suppose that there exist operator measure...and operator $\mathcal %L_t$...    
%\end{corollary}
%}

The next result is a quenched LLT.
Let $\pi_1:\Sigma\times X\to X$ be the natural projection. Define $S_n:\Sigma\times X\to\mathbb R$ by
$$
S_n A=\sum_{j=0}^{n-1}A\circ \pi_1\circ F_1^j.
$$
Denote $\wh A(x,y)=A(x)-A(y)$, $\pi_2(\omega,x,y)=(x,y)$ and $S_n \wh A:\Sigma:\times X\times X\to\mathbb R$ by
$$
S_n \wh A=\sum_{j=0}^{n-1}\wh A\circ\pi_2\circ F_2^j.
$$
As before, let 
$$
\varphi_n(t)=\int e^{it S_n\wh A/\sqrt n}d\widehat{\mu}^{(2)}
$$
and
$$
\phi_n(t)=\int e^{it S_nA/\sqrt n}d\widehat{\mu}.
$$
Then the classical Fourier approach for proving  CLT rates and the local CLT for both $S_n A$ and $S_n \wh A$ passes through certain asymptotic conditions on $\varphi_n(t)$ and $\phi_n(t)$. In the theorem below  we will show that the same conditions are, in fact, sufficient for the quenched CLT and the quenched local CLT.
%Let $\mathcal L_0$ and $\widehat{\mathcal L}_0$ be the duals of the Koopman operators $H\to H\circ F$ and $H\to H\circ F_2$ with respect to the measures $\widehat{\mu}$ and $\widehat{\mu}^{(2)}$, respectively. Let set $L_t(H)=L(e^{itA}H)$ and $\widehat{\mathcal L}_t(H)=\widehat{\mathcal L}_0(e^{it\widehat A}H)$. 
%The the asymptotic conditions conditions are satisfied when these operators are $C^3$ in $t$, they are quasi compact and $L_0$ and $\widehat{\mathcal L}_0$ have a spectral gap with multiplicity $1$ (see \cite{HH}).

\begin{theorem}[Non-lattice LLT]\label{LLT1}
Let the conditions of Theorem \ref{BE} be in force with $a=0,1$ and $T_n=\ln n$.  Assume also that there exist constants $\delta\in(0,1)$, $C,c>0$ and $a>2$ such that for all $\theta\in[-\delta,\delta]$,
\begin{equation}
\left|\int e^{i\theta S_n\wh A}d\widehat{\mu}^{(2)}\right|\leq Ce^{-c\theta^2 n}
\end{equation}
 and  for every $N\in\mathbb N$, 
\begin{equation}
\int_{\delta\leq |\theta|\leq N}\left|\int e^{i\theta S_n\wh A}d\widehat{\mu}^{(2)}\right|d\theta=O(n^{-a}).     
\end{equation}
Then for every continuous function $G$ with compact support or an indicator of a bounded interval we have
$$
\lim_{n\to\infty}\sup_{u\in\mathbb R}\left|\sqrt{2\pi n}\int G(S_n^\omega A-u)d\widehat{\mu}_\omega-e^{-\frac{u^2}{2n}}\int G(x)dx\right|=0
$$
$\mu$-almost surely. 
\end{theorem}
\begin{remark}
In applications, the conditions of Theorem \ref{LLT1} can be verified through  quasi-compactness  of appropriate complex operators, together with a spectral gap property and irreducibility of the function $\wh A$. Yet we decided to state Theorem \ref{LLT1} under these weaker conditions.   
\end{remark}

The following result complements Theorem \ref{LLT1} and it proven essentially in the same way (see the proof of \cite[Theorem 2.2.3]{HK}).
\begin{theorem}[Lattice LLT]\label{LLT2}
Let conditions of Theorem \ref{BE} are in force with $a=0$ and $T_n=\ln n$.  Assume also that there exist constants $\delta\in(0,1)$, $C,c>0$ and $a>2$ such that 
\begin{equation}
\left|\int e^{i\theta S_n\wh A}d\widehat{\mu}^{(2)}\right|\leq Ce^{-c\theta^2 n}
\end{equation}
for all $\theta\in[-\delta,\delta]$ and
\begin{equation}
\int_{\delta\leq |\theta|\leq \pi}\left|\int e^{i\theta S_n\wh A}d\widehat{\mu}^{(2)}\right|=O(n^{-a}).     
\end{equation}
Then for every continuous function $G$ with compact support or an indicator of a bounded interval we have
$$
\lim_{n\to\infty}\sup_{u\in\mathbb Z}\left|\sqrt{2\pi n}\int G(S_n^\omega A-u)d\widehat{\mu}_\omega-e^{-\frac{u^2}{2n}}\sum_{k\in\mathbb Z}G(k)\right|=0
$$
$\mu$-almost surely.     
\end{theorem}

\begin{remark}\label{Rem2}
 We note that using standard methods (see for example \cite{HH}) the conditions in Theorem \ref{CLT} yield the CLT for the sum $S_nA$ and the conditions in Theorem \ref{LLT1} and \ref{LLT2} with $A$ instead of $\widehat{A}$ yield the corresponding LLT for the sums $S_nA$ (with respect to the measure $\wh\mu$). This is a new result in our applications in Sections \ref{App1}-\ref{App6}.   
\end{remark}
\subsection{Effective quenched decay of correlations}\label{Corr}
Let $\|\cdot\|$ be a norm on functions on $X$.
\begin{theorem}[Eff. quenched DOC]\label{DEC}
Suppose that $\|\cdot\|_\infty\leq \|\cdot\|$.
Assume also that the two point motion is exponentially mixing in the sense that there exist $C>0$ and $\zeta\in(0,1)$ such that for all functions of the form $H(x,y)=h(x)h(y)$ and $G(\omega,x,y)=(g(x)-\widehat{\mu}_\omega(g))(g(y)-\widehat{\mu}_\omega(g))$ with $\|g\|,\|h\|\leq 1$ we have 
$$
\left|\int G\cdot (H\circ F_2^n)d\widehat{\mu}^{(2)}\right|\leq C\zeta^n. 
$$
Then, for all functions $g,h:X\to\mathbb R$  such that  $\|g\|,\|h\|\leq 1$, for every $1\leq p<\infty$ there exist a constant $c_p>0$ and a random variable $R_p\in L^p(\mu)$ we have
$$
\left|\int g\cdot (h\circ f_\omega^n)d\widehat{\mu}_\omega-\left(\int g\,d\widehat{\mu}_\omega\right)\left(\int h\circ f_\omega^n\,d\wh{\mu}_{\omega}\right) \right|\leq R_p(\omega)e^{-c_p n}.
$$
Moreover, suppose that $\wh\mu$ is $F_1$ invariant or $\wh\mu_{\omega}$ and $X_\omega$ does not depend on $\omega$ and $\nu(g)=0$, where $\nu$ is the projection of $\wh\mu$ to $X$. Let $H(x,y)=h(x)h(y)$ and $\tilde G(x,y)=g(x)g(y)$. If
$$
\left|\int \tilde G\cdot (H\circ F_2^n)d\widehat{\mu}^{(2)}\right|\leq C\zeta^n. 
$$
then for all $1\leq p<\infty$ there is a random variable $R_p\in L^p(\mu)$ such that 
$$
\left|\int g\cdot (h\circ f_\omega^n)d\widehat{\mu}_\omega\left(\int h\circ f_\omega^n\,d\wh{\mu}_{\omega}\right) \right|\leq R_p(\omega)e^{-c_p n}.
$$
\end{theorem}
\begin{remark}
Compared with iid case considered in \cite{DimaJonEquivalence} we are able to get much stronger integrability properties on $R_p$, and get the results in a more general non iid framework.
\end{remark}

The conditions of the first result in Theorem \ref{DEC}  involve  the functions $G(\omega,x,y)$. The decay rate on the $F_2$ level basically means that the means $\widehat{\mu}_\omega(g)$ are sufficiently smooth functions of $\omega$. This is clearly the case when $\widehat{\mu}_\omega$ does not depend on $\omega$ (but in this case we can use the second part). For instance $X_\omega$ can all coincide with $X$ which can be a given compact Riemannian manifold. In that case we take $\widehat{\mu}_\omega=\mathrm{Vol}$. However, it is desirable to prove some results when $\widehat{\mu}_\omega$ is truly random, since this is usually the case when $\widehat{\mu}$ is $F$ invariant. Suppose that the measures $\widehat{\mu}_\omega$ are SRB in the sense that $X_\omega=X$ is a compact manifold and for $\mu$-a.a. $\omega$  for every function $g:X\to\mathbb R$ with $\|g\|\leq 1$ we have 
\begin{equation}\label{Cond}
\left|\int g\circ f_{T^{-n}\omega}^nd\mathrm{Vol}-\widehat{\mu}_\omega(g)\right|\leq Ce^{-\lambda n} 
\end{equation}
where $\mathrm{Vol}$ is the normalized volume measure, $C,\lambda>0$ are constants and 
$$
f_{T^{-n}\omega}^n=f_{T^{-1}\omega}\circ\cdots\circ f_{T^{-n}\omega}.
$$
To simplify the arguments it will convenient to assume that $(\Sigma,\mathcal F,\mu,T)$ is a topologically mixing subshift of finite type, $\mu$ is some Gibbs measure (see \cite{Bowen}) and that  $f_\omega$ depends only on the $0$-th coordinate of $\omega$.
\begin{prop}\label{ppp}
 In the above circumstances there are constants $C_0>0$ and $\alpha\in(0,1]$ such that for every function $g$ with $\|g\|\leq 1$ the map $\omega\to\widehat{\mu}_\omega$ is H\"older continuous with exponent $\alpha$ and H\"older constant not exceed $C$. More precisely,
 $$
|\widehat{\mu}_\omega(g)-\wh{\mu}_{\omega'}(g)|\leq 2Ce^{-k\lambda}
 $$
 if $\omega$ and $\omega'$ has the same $j$-th coordinate for $-k\leq j\leq 0$. 
\end{prop}
\begin{proof}
Let us take two points $\omega,\omega'$ and suppose that $\omega_j=\omega'_j$ for all $-k\leq j\leq 0$.  Then
$$
|\widehat{\mu}_\omega(g)-\mu_{\omega'}(g)|\leq 2Ce^{-k\lambda}+\left|\int g\circ f_{T^{-k}\omega}^{k}d\mathrm{Vol}-\int g\circ f_{T^{-k}\omega'}^{k}d\mathrm{Vol}\right|.
$$
Notice that the second term on the above right hand side vanishes since $f_{T^{-k}\omega}^{k}$ depends only on $\omega_j$ for $-k+1\leq j\leq 0$.
\end{proof}

\subsection{Effective exponential concentration inequalities}
\begin{theorem}\label{Exp}
Suppose that there exist $C,c>0$ such that for every $\epsilon>0$ small enough and all $n\in\mathbb N$ we have 
$$
\widehat{\mu}(S_nA\geq \epsilon n)\leq Ce^{-c\epsilon^2n}.
$$
Then for every $\epsilon>0$ and $1\leq p<\infty$ there exist a random variable $K_p\in L^p(\mu)$  such that for $\mu$-a.a. $\omega$ and all $n\in\mathbb N$ we have
$$
\widehat{\mu}_\omega(S_n^\omega A\geq \epsilon n)\leq K_p(\omega)e^{-\epsilon^2\frac{c}{p+2}n}. 
$$
\end{theorem}

\section{Proof of Theorems \ref{BE}, \ref{LLT1} and \ref{LLT2}}
\begin{proof}[Proof of Theorem \ref{BE}]
The first step in the proof is partially based on an idea in \cite{DimaJonEquivalence}.
 Take some $\alpha>1/(\eta-\zeta)$ and let take some $m,n$ such that $m^\alpha\leq n<(m+1)^\alpha$. Then  since $\|S_n^\omega A-S_k^\omega A\|_{L^1(\wh\mu_\omega)}\leq B_\omega n^{\delta_0/2}(n-k)^{1/2}$, $n>k$, almost surely we see that there is a constant $C_\omega>0$ such that 
 $$
|\varphi_{\omega,n}(t)-\varphi_{\omega,m^\alpha}(t)|\leq|t|\left\|S_n^\omega A/\sqrt n-S_{m^\alpha}^\omega/\sqrt{m^\alpha}\right\|_{L^1(\widehat{\mu}_\omega)}
$$
$$
\leq  C_\omega|t|n^{\delta_0/2}\left(n^{-1/2}n^{\frac{\alpha-1}{2\alpha}}+n^{-1/\alpha}\right)\leq 2C_\omega|t|n^{-1/(2\alpha)+\delta_0/2}
 $$
 where we used that $n-m^\alpha\leq Cm^{\alpha-1}$ and that $1/\sqrt x=1/\sqrt y-\frac{y-x}{\sqrt{xy}(\sqrt x+\sqrt y
 )}$ for all $x,y>0$. Therefore, 
  \begin{equation}\label{Main}
  \int_{|t|\leq T_n}\frac{|\varphi_{\omega,n}(t)-e^{-t^2/2}|}{|t|^a}dt\leq 
\int_{|t|\leq T_n}\frac{|\varphi_{\omega,m^\alpha}(t)-e^{-t^2/2}|}{|t|^a}dt    
  \end{equation}
 $$ 
+C_aT_n^{2-a}C_\omega n^{-1/(2\alpha)+\delta_0/2}\leq \int_{|t|\leq T_{(m+1)^\alpha}}\frac{|\varphi_{\omega,m^\alpha}(t)-e^{-t^2/2}|}{|t|^a}dt
 $$
 $$
+C_aT_n^{2-a}C_\omega n^{-1/(2\alpha)+\delta_0/2}.
 $$
 Next, take some $\epsilon_m>0$ and define
 $$
\Gamma_m(\omega)=\int_{\epsilon_m\leq |t|\leq T_{(m+1)^\alpha}}\frac{|\varphi_{\omega,m^\alpha}(t)-e^{-t^2/2}|}{|t|^a}dt.
 $$
 Then by the Cauchy-Schwartz inequality,
 $$
\Gamma_m^2(\omega)\leq 2 T_{(m+1)^\alpha}\int_{\epsilon_m\leq |t|\leq T_{(m+1)^\alpha}}\epsilon_m^{-a}\frac{|\varphi_{\omega,m^\alpha}(t)-e^{-t^2/2}|^2}{|t|^a}dt.
 $$
 Thus, 
 $$
\int \Gamma_m^2(\omega)d\mu(\omega)\leq  2 T_{(m+1)^\alpha}\epsilon_m^{-a}\int_{\epsilon_m\leq |t|\leq T_{(m+1)^\alpha}}\frac{\int|\varphi_{\omega,m^\alpha}(t)-e^{-t^2/2}|^2 d\mu(\omega)}{|t|^a}dt.
 $$
Now,
$$
\int|\varphi_{\omega,m^\alpha}(t)-e^{-t^2/2}|^2 d\mu(\omega)=\varphi_{m^\alpha}(t)+e^{-t^2}-2e^{-t^2/2}\int\mathrm{Re}(\varphi_{\omega,m^\alpha}(t))d\mu(\omega).
$$
Notice that
$$
\int_{\epsilon_m\leq |t|\leq T_{(m+1)^\alpha}}|t|^{-a}2e^{-t^2/2}\int\mathrm{Re}(\varphi_{\omega,m^\alpha}(t))dt\mu(\omega)
$$
$$
=\int_{\epsilon_m\leq |t|\leq T_{(m+1)^\alpha}}2e^{-t^2/2}|t|^{-a}\mathrm{Re}\left(\int\varphi_{\omega,m^\alpha}(t)dt\mu(\omega)\right)
$$
$$
=2\int_{\epsilon_m\leq |t|\leq T_{(m+1)^\alpha}}|t|^{-a}e^{-t^2/2}\mathrm{Re}(\phi_{m^\alpha}(t))dt
$$
$$
=2\int_{\epsilon_m\leq |t|\leq T_{(m+1)^\alpha}}|t|^{-a}e^{-t^2/2}\mathrm{Re}(\phi_{m^\alpha}(t)-e^{-t^2/2})dt
$$
$$
+2\int_{\epsilon_m\leq |t|\leq T_{(m+1)^\alpha}}|t|^{-a}e^{-t^2}dt
$$
$$
\leq2\int_{\epsilon_m\leq |t|\leq T_{(m+1)^\alpha}}|t|^{-a}|\phi_{m^\alpha}(t)-e^{-t^2/2}|dt+2\int_{\epsilon_m\leq |t|\leq T_{(m+1)^\alpha}}|t|^{-a}e^{-t^2}dt.
$$
Therefore,  by  \eqref{CF2} and the above estimates,
$$
\int(\Gamma_m(\omega))^2d\mu(\omega)
$$
$$
\leq  CT_{(m+1)^\alpha}\epsilon_m^{-a}\left(m^{-\alpha\eta}+
\int_{\epsilon_m\leq |t|\leq T_{(m+1)^\alpha}}|t|^{-a}|\varphi_{m^\alpha}(t)-e^{-t^2}|dt\right).
$$
Using \eqref{CF1} conclude  that 
$$
\int \Gamma_m^2(\omega)d\mu(\omega) \leq C_0 T_{(m+1)^\alpha}\epsilon_m^{-a}m^{-\alpha\eta}
$$
for some constant $C_0$.
Thus, using that $T_k=O(k^{\zeta})$, taking $\epsilon_m=m^{-r}$ and taking $\alpha$ such that $\alpha(\eta-\zeta)>1+ra$ we see that 
$$
\int \Gamma_m^2(\omega)d\mu(\omega)=O(m^{-1-\epsilon})
$$
where $\epsilon=\alpha(\eta-\zeta)-ra$. Thus, by the Markov inequality and the Borel Cantelli lemma we see that for every $0<\delta<\epsilon/2$,
$$
 \Gamma_m(\omega)=O(m^{-\delta})
$$
almost surely. Recalling \eqref{Main}, we thus conclude that 
$$
\int_{|t|\leq T_n}\frac{|\varphi_{\omega,n}(t)-e^{-t^2/2}|}{|t|^a}dt
$$
$$
\leq 
\int_{|t|\leq n^{-r\alpha}}\frac{|\varphi_{\omega,n}(t)-e^{-t^2/2}|}{|t|^a}dt+C_aC_\omega n^{(2-a)\zeta}n^{-1/(2\alpha)+\delta_0/2}+O(n^{-\delta/\alpha}).
$$
Now, notice that for $|t|\leq n^{-r\alpha}$ we have $|e^{-t^2/2}-1|\leq C|t|$ and 
$$
|\varphi_{\omega,n}(t)-1|\leq C_\omega n^{\delta_0/2}|t|.
$$
Therefore,
$$
\int_{|t|\leq n^{-r\alpha}}\frac{|\varphi_{\omega,n}(t)-e^{-t^2/2}|}{|t|^a}dt\leq C_\omega n^{\delta_0/2}\int_{|t|\leq n^{-r\alpha}}|t|^{1-a}dt\leq n^{-r\alpha(2-a)+\delta_0/2}. 
$$

The last part of the theorem follows by the Esseen inequality which asserts that 
$$
\sup_{t\in\mathbb R}|\widehat{\mu}_\omega\{S_n^\omega A\leq t\sqrt n\}-\Phi(t)|\leq \int_{|t|\leq T_n}\frac{|\varphi_{\omega,n}(t)-e^{-t^2/2}|}{|t|}dt+C/T_n
$$
for some absolute constant $C$.
\end{proof}

\begin{proof}[Proof of Theorem \ref{LLT1}]
Take $0<b<1$. Then, by applying \cite[Theorem 10.7]{Breiman} with the set of functions of the form $h(x)=\left(\frac{\sin(\lambda x)}{\lambda x}\right)^2, \lambda>0$
it is enough to prove the theorem for all continuous complex valued functions $G$ on $\mathbb R$ such that 
$$
\int_{-\infty}^{\infty} |x|^b|G(x)|dx<\infty
$$
which have  a Fourier transform with compact support. 

Fix such a function $G$.
Arguing like in the beginning of the proof of \cite[Theorem 2.2.3]{HK} (see \cite[Eq. (2.2.8)]{HK}) we have 
$$
\left|\sqrt{2\pi n}\int G(S_n^\omega A-u)d\widehat{\mu}_\omega-e^{-\frac{u^2}{2n}}\int G(x)dx\right|\leq \sum_{j=1}^{4}I_j(\omega,n) 
$$
where 
$$
I_1(\omega,n)=\int_{|\theta|\leq \ln n}|\varphi_{\omega,n}(\theta)\widehat G(\theta/\sqrt n)-e^{-\theta^2/2}G(0)|d\theta,
$$
$$
I_2(\omega,n)=\left|\int G(x)dx\right|\int_{|\theta|\geq \ln n}e^{-\theta^2/2}d\theta,
$$
$$
I_3(\omega,n)=\int_{\ln n\leq |\theta|\leq \delta\sqrt n}|\varphi_{\omega,n}(\theta)|d\theta,
$$
$$
I_4(\omega,n)=\int_{\delta\sqrt n\leq|\theta|\leq N\sqrt n}|\varphi_{\omega,n}(\theta)|d\theta=
\sqrt n\int_{\delta\leq |\theta|\leq N}|\widehat{\mu}_\omega(e^{i\theta S_n^\omega A})|d\theta
$$
where $N$ satisfies that $\widehat G$ is supported on $[-N,N]$. 

Next, using Theorem \ref{BE} with $a=0$ and that $|\widehat G(x)-\widehat G(y)|\leq C|x-y|^b$ (which follows since $x\to e^{ix}$ is H\"older continuous with exponent $b$) we see that 
$$
\lim_{n\to\infty}\int_{|\theta|\leq \ln n}|\varphi_{\omega,n}(\theta)\widehat G(\theta/\sqrt n)-e^{-\theta^2/2}G(0)|d\theta=0
$$
almost surely. Thus, $\lim_{n\to\infty}I_{1}(\omega,n)=0$.

Next, by the Cauchy-Schwartz inequality,
$$
(I_{3}(\omega,n))^2\leq C\sqrt n \int_{\ln n\leq |\theta|\leq \delta\sqrt n}|\varphi_{\omega,n}(\theta)|^2d\theta
$$
and so 
$$
\int (I_{3}(\omega,n))^2d \mu(\omega)\leq C\sqrt n \int_{\ln n\leq |\theta|\leq \delta\sqrt n}|\varphi_{n}(\theta)|d\theta=O(\sqrt n e^{-c\ln^2(n)})=O(n^{-2}).
$$
Therefore by the Markov inequality and the Borel Cantelli lemma we see that 
$$
I_{3}(\omega,n)\to 0, \text{a.s.}
$$
Next, by the Cauchy-Schwartz inequality, for all $N\in\mathbb N$,
$$
\left(\int_{\delta\leq |\theta|\leq N}\left|\int e^{i\theta S_n^\omega A}d\widehat{\mu}_\omega\right|\right)^2d\theta\leq C_N\int_{\delta\leq |\theta|\leq N}\left|\int e^{i\theta S_n^\omega A}d\widehat{\mu}_\omega\right|^2d\theta
$$
and so 
$$
\int \left(\int_{\delta\leq |\theta|\leq N}\left|\int e^{i\theta S_n^\omega A}d\widehat{\mu}_\omega\right|d\theta\right)^2d\mu(\omega)
$$
$$
\leq C_N\int\int_{\delta\leq |\theta|\leq N}\left|\int e^{i\theta S_n^\omega A}d\widehat{\mu}_\omega\right|^2d\theta \,d \mu(\omega)
$$
$$
=\int_{\delta\leq |\theta|\leq N}\left(\int e^{i\theta S_n\wh A/\sqrt n}d\widehat{\mu}^{(2)}\right)d\theta=O(n^{-a}).
$$
Thus, by the Markov inequality, for $0<\epsilon<a-2$,
$$
\mu\left(\int_{\delta\leq |\theta|\leq N}\left|\int e^{i\theta S_n^\omega A}d\widehat{\mu}_\omega\right|d\theta\geq n^{-1/2-\epsilon/2}\right)=O(n^{-(a-1-\epsilon)}).
$$
Hence by the Borel Cantelli lemma, 
$$
\int_{\delta\leq |\theta|\leq N}\left|\int e^{i\theta S_n^\omega A}d\widehat{\mu}_\omega\right|d\theta=O(n^{-1/2-\epsilon/2})
$$
for $\mu$-a.e. $\omega$. We conclude that 
$$
I_{4}(\omega,n)=o(n^{-\epsilon/2})
$$
almost surely. Finally, it is clear that $I_2(\omega,n)\to 0$ as $n\to\infty$, and the proof of the theorem is complete.
\end{proof}
\begin{proof}[Proof of Theorem \ref{LLT2}]
Arguing like in the  the proof of \cite[Theorem 2.2.3]{HK} in the lattice case 
the proof is similar to the proof of Theorem \ref{LLT1}. The only big difference is that in the definition of $I_4$ we need to take $N=\pi$.   
\end{proof}

\section{Proof of Proposition \ref{CondProp}}

 Denote 
$$
C_n(\omega)=\int A\cdot (A\circ f_\omega^n)d\widehat{\mu}_\omega.
$$
Take some $c>0$ and denote
$$
A_{n}=\{\omega: |C_{n}(\omega)|\geq e^{-cn}\}.
$$
Then  there is $\zeta\in(0,1)$ such that
$$
\int (C_n(\omega))^2d\mu(\omega)=\int \tilde A\cdot (\tilde A\circ F_2^n)d\widehat{\mu}^{(2)}=O(\zeta^n).
$$ 
Therefore, by the Markov inequality we can choose $c$ such that 
$$
\mu(A_{n})=O(e^{2cn}\zeta^n)=O(\gamma^n),\, \gamma=\sqrt\zeta.
$$ 
 Therefore, by applying  the Borel Cantelli lemma we see that for $\mu$-almost all $\omega$ there exists $N_2(\omega)$ such that if $n\geq N_2(\omega)$ then  
$$
|C_{n}(\omega)|\leq e^{-cn}.
$$
Let us write $\gamma=e^{-a}, a>0$. Fix some $p\geq 1$. By decreasing $c$ we can always assume that $pc<a$. 
Thus, using the trivial bound $|C_{n}(\omega)|\leq (\sup|A|)^2$,
$$
C_{n}(\omega)\leq \sup|A|e^{cN_2(\omega)}e^{-cn}:=K_2(\omega)e^{-cn}.
$$
Notice that, in fact, we can always assume that $N_2(\omega)$ is the minimal number $N$ such that 
if $n\geq N$ then  
$$
|C_{n}(\omega)|\leq e^{-cn}.
$$
Therefore, for all $k\in\mathbb N$ we have
$$
\{\omega: N_2(\omega)=k+1\}\subset \{C_k(\omega)\geq e^{-ck}\}
$$
and by the Markov inequality, so with some constant $A>0$,
$$
\mu(N_2=k+1)\leq \mu(|C_k|\geq e^{-ck})\leq A\gamma^k=Ae^{-ak}.
$$
Therefore,
$$
\int (K_2(\omega))^pd\mu(\omega)\leq 4^p(e^{cp}+\sum_{k=2}^\infty e^{cpk}e^{-a(k-1)})<\infty
$$
where the last estimate uses that  $pc<a$. We thus conclude that 
$
K_2(\cdot)\in L^p(\mu).
$
Using the mean ergodic theorem it follows that 
$$
C_n(T^k\omega)=O(k^{1/p}e^{-cn})
$$
and so $\mu$-a.s. for all $k<n$ we have
$$
\wh{\mu}_\omega((S_n^\omega A-S_k^\omega A)^2)=O(n^{1/p}(n-k)) 
$$
for all $n>k$ and $p>1$.
\section{Proof of Theorems \ref{DEC} and \ref{Exp}}
\begin{proof}[Proof of Theorem \ref{DEC}]
Let us prove the first statment of the theorem. The second part follows similarly (see also the proof of Proposition \ref{CondProp}).
Note that the underlying functions $g,h$ are bounded by $1$. Denote 
$$
C_n(\omega)=\int g\cdot (h\circ f_\omega^n)d\widehat{\mu}_\omega-\left(\int g\,d\widehat{\mu}_\omega\right)\left(\int h\,d\mu_{T^n\omega}\right). 
$$
Take some $0<c<-(\ln\zeta)/4$ and denote
$$
A_{n}=\{\omega: |C_{n}(\omega)|\geq e^{-cn}\}.
$$
Then  
$$
\int (C_n(\omega))^2d\mu(\omega)=\int G\cdot (H\circ F_2^n)d\widehat{\mu}^{(2)}=O(\zeta^n)
$$ 
Therefore, by the Markov inequality we can choose $c$ such that 
$$
\mu(A_{n})=O(e^{2cn}\zeta^n)=O(\gamma^n),\, \gamma=\sqrt\zeta.
$$ 
Therefore, by applying  the Borel Cantelli lemma we see that for $\mu$-almost all $\omega$ there exists $N_2(\omega)$ such that if $n\geq N_2(\omega)$ then  
$$
|C_{n}(\omega)|\leq e^{-cn}.
$$
Let us write $\gamma=e^{-a}, a>0$. Fix some $p\geq 1$. By decreasing $c$ we can always assume that $pc<a$. 
Thus, using the trivial bound $|C_{n}(\omega)|\leq 2$,
$$
C_{n}(\omega)\leq 2e^{cN_2(\omega)}e^{-cn}:=K_2(\omega)e^{-cn}.
$$
Notice that, in fact, we can always assume that $N_2(\omega)$ is the minimal number $N$ such that 
if $n\geq N$ then  
$$
|C_{n}(\omega)|\leq e^{-cn}.
$$
Therefore, for all $k\in\mathbb N$ we have
$$
\{\omega: N_2(\omega)=k+1\}\subset \{C_k(\omega)\geq e^{-ck}\}
$$
and so 
$$
\mu(N_2=k+1)\leq \mu(|C_k|\geq e^{-ck})\leq A\gamma^k=Ae^{-ak}.
$$
Therefore,
$$
\int (K_2(\omega))^pd\mu(\omega)\leq 4^p(e^{cp}+\sum_{k=2}^\infty e^{cpk}e^{-a(k-1)})<\infty.
$$
where the last estimate uses that  $pc<a$. We thus conclude that 
$
K_2(\cdot)\in L^p(\mu).
$
\end{proof}

\begin{proof}[Proof of Theorem \ref{Exp}]
 Fix $\epsilon>0$ and $1\leq p<\infty$. Denote 
 $$
\Gamma_{n,\epsilon}(\omega)=\widehat{\mu}_\omega(S_n^\omega A\geq \epsilon n).
 $$
 Then 
 $$
\int \Gamma_{n,\epsilon}(\omega)d\mu(\omega)=\widehat{\mu}(S_nA\geq \epsilon n)\leq Ce^{-c\epsilon^2 n}.
 $$
 Take some $0<d<c$. Then by the Markov inequality,
$$
\mu(\Gamma_{n,\epsilon}\geq e^{-d\epsilon^2n})\leq Ce^{-\epsilon^2(c-d)n}
$$
 and so by the Borel Cantelli lemma there exists a random variable $N_{\epsilon,d}(\omega)$ such that if $n\geq N_{\epsilon,p}(\omega)$ then
 $$
\Gamma_{n,\epsilon}(\omega)\leq e^{-d\epsilon^2 n}.
 $$
 Note that we can always assume that $N_{\epsilon,d}(\omega)$ is the minimal positive integer $N$ such that for all $n\geq N$ we have 
  $$
\Gamma_{n,\epsilon}(\omega)\leq e^{-d\epsilon^2 n}.
 $$
 Since $\Gamma_{n,\epsilon}\leq 1$ we conclude that for all $n\in\mathbb N$,
 $$
\Gamma_{n,\epsilon}(\omega)\leq e^{dN_{\epsilon,d}(\omega)}e^{-d\epsilon^2 n}.
 $$
 Set $K(\omega)=e^{d\epsilon^2 N_{\epsilon,d}(\omega)}$. Notice that for every $k\in\mathbb N$,
 $$
\mu(N_{\epsilon,d}=k+1)\leq \mu(\Gamma_{n,\epsilon}\geq e^{-d\epsilon^2 k})\leq Ce^{-\epsilon^2(c-d)k}.
 $$
 Therefore,
 $$
\int (K(\omega))^pd\mu(\omega)\leq e^{dp\epsilon^2}+C\sum_{k=2}^{\infty}e^{dk\epsilon^2 p}e^{-\epsilon^2(c-d)(k-1)}<\infty
 $$
Hence, $K\in L^p$ if $d(p+1)<c$ (so we can take $d=\frac{c}{p+2}$), and the proof of the theorem is complete.
\end{proof}

\section{Application to effectively expanding on average random diffeomorphisms}\label{App1}
Suppose that $(\Sigma,\mathcal F,\mu,T)$ is a Bernoulli shift with a finite number of symbols $\Sigma_0$ and write $\Sigma=\Sigma_0^\mathbb Z$ and $\mu=\mu_0^\mathbb Z$ for some measure $\mu_0$ on $\Sigma_0$.
Let $M$ be a closed Riemannian manifold of dimension $d\geq 2$. Let $\{f_a\}_{a\in \Sigma_0}\subseteq \mathrm{Diff}^{1+\gamma}(M)$, $\gamma>0$.

\begin{definition}[Pressure out of equilibrium {\cite[Definition~2.1]{TDFOE-I}}]\label{DefOfPOE}
	Given a compact metric space $X$ and $\wh\phi\in\mathrm{H\ddot{o}l}(\Sigma\times X)$, we define the {\em pressure out of equilibrium}, or {\em POE} for short, as $$\widehat{P}(\widehat{\phi}):=\limsup_{n\to\infty}\frac{1}{n}\log\sup_{t\in X}\widehat{Z}_n(\widehat{\phi},t,a),$$ 
	where
	$$\widehat{Z}_n(\widehat{\phi},t,a):=\sum_{|\underline{w}|=n,w_{n-1}=a}e^{\sum_{k=0}^{n-1}\phi_{F^k_{\theta_{\underline{w}}}(t)}(T^k(\theta_{\underline{w}}))},$$
	$a\in \Sigma_0$, and $\theta_{\underline{w}}\in[\underline{w}]$ maximize $\omega\mapsto \sum_{k=0}^{n-1}\phi_{F^k_{\omega}(t)}(T^k(\omega))$.
\end{definition}

\begin{definition}[{Derivative cocycle extension \cite[Definition~9.2]{TDFOE-II}}]
We denote the {\em derivative cocycle extension} by $\wh F_L:\Sigma^+\times T^1M\to \Sigma^+ \times T^1M$, where
$$\wh F(\omega, x,\xi):=\Big(T\omega, f_{\omega},\frac{d_xf_{\omega}\xi}{|d_xf_{\omega}\xi|}\Big).$$
Define the derivative potential
$$\wh D(\omega,x,\xi):=-\log |d_{x} f_{\omega}^{-1}\xi|,$$
as a H\"older potential on the skew-product system.
\end{definition}

Let $\wh f:\Sigma\times M\to \Sigma\times M$ be the skew-product map $\wh f(\omega,x)=(T\omega,f_\omega(x))$.

\begin{definition}[{Similarity dimension \cite[Definition~9.4]{TDFOE-II}}]
	The {\em similarity dimension} associated with $(\Sigma\times M,\wh f)$ and $\mu$, where $\mu$ is the Gibbs state of $\psi$, is
	$$d^*:=\sup\Big\{\beta>0: \wh P(\psi- \beta\cdot \wh D)<0\Big\},$$
	and is defined as $0$ if the supremum is over an empty set.
\end{definition}

\begin{lemma}[{\cite[Lemma~9.6]{TDFOE-II}}]
	$$d^* \leq d.$$
\end{lemma}

\begin{definition}[{Effective expansion on average \cite[Definition~9.9]{TDFOE-II}}]
     We say that $(\Sigma\times M,\wh f)$ admits {\em effective expansion on average w.r.t. $\mu$} if
    \begin{equation*}\label{UEAEq}
%      \wh P(\psi- d\cdot \wh D)<0.
d^*>d-2.
    \end{equation*}
We may also specify and say that  $(\Sigma\times M,\wh f)$ admits {\em effective expansion on average with a parameter $\chi>0$} if $\inf_{\kappa\in (d-2,d^*)}\wh P(\psi- \kappa\cdot \wh D)<-\chi$.
\end{definition}

\begin{remark}
    Equivalently to effective expansion on average is, $\exists \kappa>d-2$ s.t. 
    $$\limsup\frac{1}{n}\log \max_{(x,\xi)\in T^1M}\int |d_xf_\omega^n\xi|^{-\kappa}d\mu<0.$$
\end{remark}

\begin{remark}
	 When $f_\omega=f_{(\omega_i)_{i\leq0}}$, we are in the setting of \cite[Theorem~6.7]{TDFOE-I} (the variational principle), and so it is easy to observe that the effective expansion on average is a $C^1$-open condition in $\{f_\omega\}_{\omega\in \Sigma}$ and $C^0$-open in $\psi$ (see \cite[Lemma~9.8]{TDFOE-II}).
\end{remark}

\begin{remark}
    The condition of {\em co-expansion on average} (see \cite{JonDima}) for conservative systems (i.e. $f_\omega$'s are volume preserving) implies effective expansion on average; Moreover with $d^*=d$. See \cite[\textsection~B]{TDFOE-II}.
\end{remark}

Denote the {\em two points dynamics} by
$$\wh f^{(2)}(\omega,x,y)=(T\omega, f_\omega x, f_\omega y).$$ 

\begin{prop}\label{onePointTwoPoints}
    Effective expansion on average implies effective expansion on average for the two points dynamics: 
     $$\limsup\frac{1}{n}\log \max_{(x,\xi),(y,\eta)\in TM,|\xi|^2+|\eta|^2=1}\int \Big(|d_xf_\omega^n\xi|^2+|d_yf_\omega^n\eta|^2\Big)^{-\frac{\kappa}{2}}d\mu<0.$$
\end{prop}
\begin{proof}
The statement is trivial if $|\xi|=0$ or $|\eta|=0$. Assume otherwise, then by Jensen's inequality
  \begin{align*}
      &\int \Big(|d_xf_\omega^n\xi|^2+|d_yf_\omega^n\eta|^2\Big)^{-\frac{\kappa}{2}}d\mu\\
      =&\int \Big(|d_xf_\omega^n\frac{\xi}{|\xi|}|^2\cdot |\xi|^2+|d_yf_\omega^n\frac{\eta}{|\eta|}|^2\cdot |\eta|^2\Big)^{-\frac{\kappa}{2}}d\mu\\
      \leq& \int |d_xf_\omega^n\frac{\xi}{|\xi|}|^{-\kappa}\cdot |\xi|^2+|d_yf_\omega^n\frac{\eta}{|\eta|}|^{-\kappa}\cdot |\eta|^2d\mu\\
      \leq & (|\xi|^2+|\eta|^2)\max_{(z,\zeta)\in T^1M}\int |d_zf_\omega^n\zeta|^{-\kappa}d\mu.
  \end{align*}  
\end{proof}

Let $\Sigma^+:=\{(\omega_i)_{g\geq0}: \omega\in \Sigma\}$ be the one-sided shift, assume that $F_\omega=F_{(\omega_i)_{i\geq0}}$, and let $\wh F_L:\Sigma^+\times X\to \Sigma^+\times X$ be defined by $\wh F_L(\omega^+, t)=(T_L\omega^+,F_{\omega^+}(t))$.

\begin{definition}[{Semi-Ruelle Operator \cite[Definition~3.7]{TDFOE-II}}]
Given $\wh\phi\in C(\Sigma^+)$, and $\wh F_L:\Sigma^+\times X \to \Sigma^+ \times X$ s.t. $F_\omega=F_{(\omega_i)_{i\geq0}}$,  we define the associated {\em semi-Ruelle operator}, $\L_{\wh\phi}:C(\Sigma^+\times X)\to C(\Sigma^+\times X) $ by
\begin{equation}\label{defOfRuelleEq}
	(\L_{\wh\phi}g)(\omega^+,t):=\sum_{\wh{F}_L(\wt\omega^+,\wt t)=(\omega^+,t)}e^{\wh\phi(\wt\omega^+,\wt t)}g(\wt\omega^+,\wt t).
\end{equation}
\end{definition}

Given $s\in (0,1)$ we denote by $\Hh_{-s}(m)$ the {\em negative index Sobolev space}, where $m$ denotes the normalized Riemannian volume measure of $M$. See \cite[\textsection~9.3]{TDFOE-II} for details. 

\begin{definition}[{Averaged semi-Ruelle operator \cite[Definition~9.20]{TDFOE-II}}]
Given $n\in\mathbb{N}$	The {\em averaged ($n$-th power of the) semi-Ruelle operator} acting on $g\in L^2(m)$, is defined by 
	$$(\mathcal L_{n}g)(x):=\Big(\int \L_{\wh\phi}^nd\mu\Big)(g)(x)=\int (\mathcal{L}_{\omega^+,n}g)(x)d\mu(\omega^+).$$
	We denote $\mathcal L_{1}$ by simply $\mathcal{L}$ for short.
\end{definition}

\begin{remark}
    Note, $\wh P(\wh\phi)$ is the log of the spectral radius of $\wh{\mathcal{L}}_{\wh\phi}:C(\Sigma^-\times X)\to C(\Sigma^-\times X)$, when $\Sigma^-:=\{(\omega_i)_{i\leq-1}:\omega\in \Sigma\}$. See \cite[Lemma~3.10]{TDFOE-II}.
\end{remark}

\begin{lemma}[{i.i.d. averaged powers \cite[Lemma~9.21]{TDFOE-II}}]\label{lemlem}
	In our setting $\mu$ is a Bernoulli measure and $f_\omega=f_{\omega_0}$. Then for all $g\in L^2(m)$, for all $n\geq1$, for all $\omega\in \Sigma$, for $m$-a.e. $x\in M$,
	$$(\mathcal L_1^ng)(x)=(\mathcal L_ng)(x)=(\L_{\wh\phi}^ng)(\omega^+,x).$$
\end{lemma}

\begin{definition}[{Averaged Koopman operator \cite[Definition~9.24]{TDFOE-II}}]
Given $n\in\mathbb{N}$	The {\em averaged ($n$-th power of the) Koopman operator} acting on $g\in L^2(m)$, is defined by 
	$$(\Uu_{n}g)(x):= \int 
	g\circ f_{\omega^+}^n(x)d\mu(\omega^+).$$
	We denote $\Uu_{1}$ by simply $\Uu$ for short.
\end{definition}

\begin{remark}
As in Lemma \ref{lemlem}, if $\Sigma$ is a full-shift, $\mu$ is a Bernoulli measure, and $f_\omega=f_{\omega_0}$, then 
	$$\Uu_n=\Uu^n.$$
\end{remark}

\begin{definition}[Volume Decay of Correlations for H\"older functions]
		We say that $\wh m$ admits {\em volume decay of correlation  on H\"older functions} with $ \nu\in \mathbb{P}(\Sigma^+\times M)$ if for all $\theta>0$ small enough there exist $C,\tau>0$ s.t. 
 for all $g,h\in \mathrm{H\ddot{o}l}_\theta(M)$, for all $n\geq0$,
\begin{equation}\label{forSobLater}
    \Big|\int g\circ \wh F^n hd\wh \mu-\int g d \nu \cdot\int h dm\Big|\leq 
    a_n\|g\|_{\mathrm{H\ddot{o}l}_\theta} \|h\|_{\mathrm{H\ddot{o}l}_\theta},
\end{equation}
	where we view $g$ and $h$ as functions on $\Sigma^+\times M$ by the natural extension, $\wh\mu:=\mu\times m$, and $a_n\to 0$.
\end{definition}

\begin{remark}
    Note, assuming that there exist $C,\tau>0$ s.t. for all $g,h\in \mathrm{H\ddot{o}l}_\theta(M)$ with $\int h dm=0$, for all $n\geq0$,
\begin{equation*}
    \Big|\int g\circ \wh F^n hd\wh \mu\Big|\leq Ce^{-\tau n}
    \|g\|_{\mathrm{H\ddot{o}l}_\theta} \|h\|_{\mathrm{H\ddot{o}l}_\theta},
\end{equation*}
Implies volume decay of correlations with some $ \nu\in \mathbb{P}(\Sigma^+\times M)$. See \cite{ExpFastVolLimits}.
\end{remark}

\begin{remark}
    Note, when $\wh\phi(\omega,x)=\psi(\omega)+\log \Jac_x(f_\omega^{-1})$ (where $\mu$ is $\psi$-Gibbs), we get that $\wh{\mathcal{L}}_{\wh\phi}$ is the $L^2(\wh\mu)$-dual of $\wh{\mathcal{U}}$, the Koopman operator of $\wh F$.
\end{remark}

Recall that $m$ denotes the normalized Riemannian volume measure of $M$.
\begin{theorem}[Complex Lasota-Yorke]\label{CLY}
Assume that $(\Sigma\times M,\wh f)$ effectively expands on average with a parameter $\chi>0$, admits volume decay of correlations for H\"older functions with a measure $ \nu$, and that $\mu$ is a Bernoulli measure with $f_\omega=f_{\omega_0}$. Then %$\Ll_t:\Hh_{-s}(m)\to \Hh_{-s}(m)$
$\Uu_t:\Hh_{s}(m)\to \Hh_{s}(m)$, for some $s\in(0,1)$, satisfies for all $n$ large enough
   $$\|\Uu_t^n g\|_{s}\leq B^n_t\|g\|_{%s-\epsilon
   L^2(m)}+ B_te^{-\chi n}\|g\|_{s},$$
   for a constant $B_t>0$ %for all $\epsilon>0$ small enough 
   where% $\Ll_t$ is the dual of 
   $$\Uu_t g=%\int e^{it A\circ f_\omega(x)}g\circ f_\omega d\mu
   e^{itA}\Uu(g),$$
   where $A\in \mathrm{H\ddot{o}l}_\theta(M)$ with $\theta>s$, and $t\in \mathbb{R}$.
\end{theorem}
\begin{proof}
In \cite[Theorem~10.7]{TDFOE-II} the author proves that under the conditions above, for some $s\in (0,1)$% and $\epsilon\in (0,s)$
, $\Uu:\Hh_{s%-\epsilon
}(m)\to \Hh_{s%-\epsilon
}(m)$ can be written as 
$$\Uu g= 1\cdot \int gd\nu+R g,$$
where $R1=0$ and $\|R^n\|_{\Hh_{s%-\epsilon
}(m)\to\Hh_{s%-\epsilon
}(m)}\leq C e^{-\chi n}$, and $\nu\in \mathbb{P}(M)$ is an element of $\Hh_{-s%+\epsilon
}(m)$. %Similarly, by \cite[Theorem~10.7]{TDFOE-II}, for some $\epsilon>0$ s.t. $s-\epsilon>0$, one can conclude that in fact $\wh\mu\in \Hh_{-s+\epsilon}(m)$.

Let $g\in \Hh_s(m)$, then
\begin{align}\label{preNormBound}
   \|\Uu_t^ng\|_s^2=&\|\Uu_t^n g\|_{L^2(m)}^2+ \mathrm{Var}_s^2\Big(\Uu_t^ng\Big)=\|\Uu^n g\|_{L^2(m)}^2+ \mathrm{Var}_s^2\Big(\Uu_t^ng\Big)\nonumber\\
   \leq &\|g\|_{L^2(m)}^2%(\|\wh\mu\|_{-s+\epsilon}+Ce^{-\chi n})^2
   \|\Uu\|_{L^2(m)\to L^2(m)}^{2n}+ \mathrm{Var}_s^2\Big(\Uu_t^ng\Big).
\end{align}
We continue to bound the variation term,
\begin{align}\label{normBound}
&\mathrm{Var}_s^2\Big(\Uu_t^ng\Big)\nonumber\\
\leq&2\int \int \int \frac{|e^{it S_n^\omega A(x)-it S_n^\omega A(y)}-1|^2|g(f_\omega^n(x))|+|g(f_\omega^n(x))-g( f_\omega^n(y))|^2}{d(x,y)^{d+2s}}dmdmd\mu\nonumber\\
\leq &2\int\int  |g(f_\omega(x))|^2\int \frac{|e^{it S_n^\omega A(x)-it S_n^\omega A(y)}-1|^2}{d(x,y)^{d+2s}}dm(y)dm(x)d\mu\nonumber\\
+&2\mathrm{Var}_s^2\Big(\Uu^n(g-\int gd\nu)\Big)\nonumber\\
\leq &2\int\int  |g(f_\omega(x))|^2\int \frac{|e^{it S_n^\omega A(x)-it S_n^\omega A(y)}-1|^2}{d(x,y)^{d+2s}}dm(y)dm(x)d\mu\nonumber\\
+&2\|R^n\|_{\Hh_s(m)\to\Hh_s(m)}^2\|g\|_s^2 \equiv \mathrm{I}+\mathrm{II}.
\end{align}
It is clear that 
\begin{equation}\label{IIbound}
    \mathrm{II}\leq 2\|R^n\|_{\Hh_s(m)\to\Hh_s(m)}^2\|g\|_s^2\leq\|g\|_s^2 2Ce^{-2\chi n}.
\end{equation}
We continue to bound $\mathrm{I}$: Let $\delta>0$ s.t. $\theta-\delta>s+\delta$, and divide $\mathrm{I}$ into integration over $E_n^\delta:=[d(x,y)\leq M_f^{-\frac{1}{\delta}2n}]$ and its complement. Over $E_n^\delta$,
\begin{align*}
    |t S_n^\omega A(x)-t S_n^\omega A(y)|\leq &|t| \|A\|_{\mathrm{H\ddot{o}l}_\theta}d(x,y)^\theta nM_f^n\\
    \leq &|t| \|A\|_{\mathrm{H\ddot{o}l}_\theta}d(x,y)^{\theta-\delta} M_f^{-2n}nM_f^n\leq C_A |t| d(x,y)^{\theta-\delta}.
\end{align*}
And so, the integral of $\mathrm{I}$ over $E_n^\delta$ is bounded by
$$(2C_A|t|)^2\int \frac{1}{d(x,y)^{d-2\delta}}dm(y)=:C_{A,t}'.$$
Next, the integral of $\mathrm{I}$ over the complement of $E_n^\delta$ is bounded by,
$$M_f^{\frac{d+2}{\delta}2n}(1+e^{2n\|A\|_\infty})^2\leq D_{A,t,f,\delta,d}^n,$$
for some constant $D_{A,f}>0$. Thus in total,
\begin{equation}\label{Ibound}
    \mathrm{I}\leq 2(C_{A,t}'+D_{A,f,\delta,d}^n)\|\Uu ^ng\|_{L^2(m)}^2\leq  \wt B^n \|g\|_{L^2(m)}^2,
\end{equation}
for some constant $\wt B=\wt B(f,A,t,\delta,d)>0$. Plugging \eqref{IIbound} and \eqref{Ibound} back in \eqref{normBound}, and then plugging \eqref{normBound} back in \eqref{preNormBound},  we get
\begin{align*}
    \|\Uu_t^ng\|_s^2\leq&\|g\|_{L^2(m)}^2%(\|\wh\mu\|_{-s+\epsilon}+Ce^{-\chi n})^2
   \|\Uu\|_{L^2(m)\to L^2(m)}^{2n}+  \wt B^n \|g\|_{%s-\epsilon
   L^2(m)}^2+\|g\|_s^2 2Ce^{-2\chi n}\\
    \leq &B^{2n}_t\|g\|_{%s-\epsilon
    L^2(m)}^2+B^2_t e^{-2\chi n}\|g\|_s^2,
\end{align*}
or some constant $B_t= B(f,A,t,\delta,d)>0$, where $ \|\Uu\|_{L^2(m)\to L^2(m)}$ is finite by \cite[Lemma~8.5]{TDFOE-II}. Then,
\begin{align*}
    \|\Uu_t^ng\|_s\leq &B^{n}_t\|g\|_{L^2(m)}+B_t e^{-\chi n}\|g\|_s.
\end{align*}
\end{proof}

\begin{cor}[Complex quasi-compactness]\label{CQC}
    For all $t\in \mathbb{R}$, for all $A\in \mathrm{H\ddot{o}l}_\theta(M)$ with $\theta>s$, the operator $\Uu_t:\Hh_{s}(m)\to \Hh_{s}(m)$ given by 
    $$\Uu_tg=e^{itA }\Uu(g),$$
    is quasi-compact with a finite-rank compact part.
\end{cor}
\begin{proof}
    Since the unit ball of $\Hh_s(m)$ embeds compactly into $L^2(m)$, this follows from Theorem \ref{CLY} by a direct application of the Henion theorem.
\end{proof}

\begin{prop}[Complex operators contraction]\label{COC} Assume volume decay of correlations with a measure $\nu$ for the two-point motion and assume effective expansion on average for the one-point motion. Then, for some $s\in (0,1)$, for all $A\in \mathrm{H\ddot{o}l}_\theta(M)$ with $\theta>s$, if for all $t\in\mathbb{R}\setminus \{0\}$ $\wh A(x,y):=A(x)-A(y)$ cannot be put in the form 
     $$\frac{e^{it \wh A}\Uu(h)}{\lambda h}=1\text { }\wh\mu\text{-a.e.},$$
     with $h\in \Hh_s(m\times m)$ and $\lambda\in\mathbb{C}^*$, then for all $t\in\mathbb{R}\setminus\{0\}$ the spectral radius of $\Uu_t$ on $\Hh_s(m\times m)$ is strictly less than $1$.
\end{prop}
\begin{proof}
    Recall Proposition \ref{onePointTwoPoints}. We therefore apply Theorem \ref{CLY} and Corollary \ref{CQC} to the two-point motion with $\wh A$. Let $t\in\mathbb{R}\setminus\{0\}$. By Corollary \ref{CQC}, $\Uu_t$ admits an eigen-functions $h_t\in \Hh_s(m)$ with an eigen-value $\lambda_t$ s.t. $|\lambda_t|$ is the spectral radius of $\Uu_t$ on $\Hh_s(m)$. Recall also that $\Uu_0=\Uu$ admits  an eigen-functions $h_0\equiv 1\in \Hh_s(m)$ with an eigen-value $1$ s.t. $1$ is the spectral radius of $\Uu$ on $\Hh_s(m)$ (by \cite[Theorem~10.7]{TDFOE-II}). We prove that $|\lambda_t|<1$. Notice that
    \begin{align*}
       |\lambda_t|^n=&| \Uu_t^nh_t|\leq \Uu^n|h_t|=\int |h_t|d\nu+O(e^{-\chi n})\leq \|\nu\|_{-s}\| |h_t| \|_s+O(e^{-\chi n})\\
       \leq&\|\nu\|_{-s}\| h_t \|_s+O(e^{-\chi n}),
    \end{align*}
 and so $|\lambda_t|\leq 1$. We may assume that $|\lambda_t|=1$, and derive a contradiction. In that case, for all $n$,
 $$|h_t|=|\Uu_t^n h_t|=|\Uu^n h_t|\leq \Uu^n |h_t|\xrightarrow[]{}\int |h_t|d\nu,$$
 which implies that $|h_t|$ is constant $\nu$-a.e. and so in particular, $h_t\neq 0$ $\nu$-a.e. Thus the identity 
 $$\frac{e^{it \wh A}\Uu(h_t)}{\lambda_t h_t}=1,$$
 is well-defined and proper, and in a contradiction to our assumption.
\end{proof}

\begin{remark}
 Note, Proposition \ref{COC} applies also to the function $A$ and the one-point motion under the same conditions and the Sobolev spaces $\Hh_s(m)$.   
\end{remark}

\begin{cor}
    Suppose that $A$ is not a coboundary with respect to $\wh F$ and  the conditions of Proposition \ref{COC} hold. Then $S_n^\omega A$ obeys the quenched CLT and the non-lattice LLT with respect to the measure $m$.
\end{cor}
\begin{proof}
Write $\wh\mu:=\mu\times m$. Using that
\begin{equation}\label{Four1}
\int e^{itS_nA}d\wh\mu=\int \mathcal U_t^n1 d\wh\mu,    
\end{equation}
and
\begin{equation}\label{Four2}
\int e^{itS_n\wh{A}}d(\wh \mu\times m)=\int \mathcal U_t^n1\,d(\wh \mu\times m),    
\end{equation}
Now, as $\mathcal U$ admits a spectral gap and $t\mapsto\mathcal U_t$ is analytic, we see that the conditions of Theorem \ref{BE} hold (see \cite{HH}). 
Indeed, by analytic perturbation theory there exist $\delta_0,C>0$ and $\rho\in(0,1)$ such that for all $t\in[-\delta_0,\delta_0]$ there are $\lambda_t\in\mathbb C, h_t\in \Hh_s, \nu_t\in \Hh_{-s}$ such that 
$$
\|\mathcal U_t^n-\lambda_t^n\cdot (\nu_t\otimes h_t)\|_{\Hh_s\to\Hh_s}\leq C\rho^n,
$$
where $\otimes$ denotes a tensor product. Moreover, $\lambda_0=1$ and $t\mapsto\lambda_t, t\mapsto h_t$ and $t\mapsto\nu_t$ are analytic in $t$. We thus conclude that the functions 
$\Lambda_n(t)=\ln\left(\int e^{itS_nA}d\wh\mu\right)$ and $\wh{\Lambda}_n(t)=\ln\left(\int e^{itS_n\wh{A}}d\wh\mu\times m\right)$ satisfy
$$
\sup_{t\in[-r_0,r_0]}\max\left(\left|\Lambda_n^{'''}(t)\right|,\left|\wh{\Lambda}_n^{'''}(t)\right|\right)\leq Cn
$$
for some $r_0,C>0$. We thus conclude that \cite[Assumption 23]{DolgHafBE0} is in force. Thus, for all $a\in[0,1]$ there exists $\zeta>0$ such that the conditions on Theorem \ref{BE} are in force, see for instance the arguments in the proof of \cite[Proposition 29]{DolgHafBE0}.

Next, note that if the spectral radius of $\mathcal U_t, t\not=0$ is smaller than $1$, then for every compact set $K\subseteq\mathbb R\setminus\{0\}$ there are $C_K,c_K>0$ such that 
$$
\Big|\int e^{itS_n\wh{A}}d(\wh \mu\times m)\Big|\leq \sup_{t\in K}\|\mathcal U_t^n\|_{\Hh_s(m\times m)\to\Hh_s(m\times m)}\leq C_Ke^{-c_K n}.
$$
Therefore, we are in the setting of both Theorem \ref{BE} and Theorem \ref{LLT1} with $\wh\mu^{(2)}:=\wh \mu\times m$, and the corollary follows.
\end{proof}

Now, as we commented in Remark \ref{Rem2} also get limit theorems for the skew products.
\begin{cor}
   Suppose that $A$ is not a coboundary with respect to $\wh F$ and  the conditions of Proposition \ref{COC} hold with $A$ instead of $\wh A$. Then $S_nA=\sum_{j=0}^{n-1}A\circ\wh F^j$, where $\wh{F}(\omega,x)=(T\omega, f_\omega x)$, obeys the quenched CLT and the non-lattice LLT with respect to the measure $\mu\times m$. 
\end{cor}

\section{Application to Non-uniformly expanding iid dynamical systems and functions with bounded variation}\label{App3}
Let us consider the following example from \cite{ANV}.  Let $A$ be a finite set and let $p=(p_i)_{i\in A}$ be a probability vector. Let us consider the Bernoulli shift $(\Sigma,\mathcal F,\mu,T)$ with $\Sigma=A^{\mathbb Z}$ and $\mu=p^{\mathbb Z}$. Let $f_\omega:[0,1]\to[0,1]$ be piecewise $C^2$ maps and let $\lambda(\omega)=\inf|f_\omega|$. We assume that $f_\omega$ depends only on the $0$-th coordinate. Define 
$$
\mathcal L_\omega g(x)=\sum_{f_\omega y=x}\frac{g(y)}{|f_\omega'(y)|}.
$$
Let $v(\cdot)$ denote the standard variation of a function on $[0,1]$. Define 
$$
P_0g(x)=\int\mathcal L_\omega g(x)d\mu(\omega).
$$
Let us take a bounded measurable function $A:[0,1]\to \mathbb R$ with bounded variation and set 
$$
P_{t}g(x)=P(e^{itg})=\int\mathcal L_\omega(ge^{it A})(x)d\mu(\omega).
$$
Then arguing like in \cite{ANV} and using that $|e^{it S_nA}|=1$ to estimate $v(S_n^\omega A\circ y)$ for an inverse branch $y$ of $f_\omega^n$  we see that if 
$$
\Lambda=\sum_{\omega_0}\frac{p_{\omega_0}}{\lambda(f_{\omega_0})}<1
$$
then there the operators $P_{t}$ are quasi compact with respect to the norm 
$$
\|g\|_{BV}=\|g\|_{L^1(m)}+v(g).
$$
Moreover, $P_0$ has multiplicity one. Let us define 
$$
\widehat{\mathcal L}_\omega g(x,x')=\sum_{f_\omega y=x, f_\omega y'=x'}\frac{g(y,y')}{|f_\omega'(y)f_\omega'(y')|}.
$$
We define $\widehat   P_t$ similarly but with the maps $f_\omega\times f_{\omega}$ and the function $\widehat A(x,y)=A(x)-A(y)$.  Then the operators $\widehat P_t$ are quasi compact and  $\widehat P_0$ has multiplicity one.
Let $\widehat{\mu}=\mu\times m$ and $\widehat{\mu}^{(2)}=\mu\times m\times m$, where $m$ is the Lebesgue measure on $[0,1]$. Then
$$
\mathbb E_{\widehat\mu}[e^{itS_nA}]=\int P_t^n\textbf{1}d\widehat{\mu}
$$
and 
$$
\mathbb E_{\widehat{\mu}^{(2)}}[e^{itS_n\widehat A}]=\int \widehat{P}_t^n\textbf{1}d\widehat{\mu}^{(2)}.
$$
Let $h$ and $\widehat h$ be the leading eigenfunctions of $P_0$ and $\widehat P_0$. Then by applying the general theory in \cite{HH} with the Markov operators 
$$
Qg(x)=P_0(gh)/h, \,\widehat{Q}g(x)=\widehat{P}_0(g\widehat{h})/\widehat{h}
$$
and the norm $\|\cdot\|_{BV}$
we see that the conditions of Theorem \ref{BE} are in force with $a=0,1$ and $T_n=o(\sqrt n)$ unless $A$ is a coboundary with respect to the skew product. Moreover, if $(X_n)$ is the stationary Markov chain with Markov operator $\widehat{Q}$ then conditions of Theorem \ref{LLT1} are in force if $\widehat A$ cannot be written as 
$$
\widehat A(X_1)=\widehat c+g(X_0)-g(X_1)+h Z(X_1)
$$
where $c$ is a constant, $g$ is a function with $\|g\|_{BV}<\infty$ and $Z$ is an integer valued function. When $A$ is integer valued then the conditions of Theorem \ref{LLT2} are in force if $A$ cannot be written in the above form with $h>\pi$.

\begin{remark}
Several other (high dimensional) examples from \cite{ANV} can be worked out, but in order not to overload the paper we decided to focus on one dimensional Lasota-Yorke maps.
\end{remark}

\section{Application to random hyperbolic dynamics with hyperbolic base maps}\label{App4}
Suppose that the random maps $f_\omega$ are uniformly hyperbolic and that $T$ is also uniformly hyperbolic.
In this case the maps $F_1$ and $F_2$ are hyperbolic. Assume that they are also topologically mixing. 
Then $F_1$ and $F_2$ admit a symbolic representation (coding) by topologically mixing subshift of finite type. In this case for every H\"older continuous function $\psi$ on $\Sigma\times X$ there exists a unique $F_1$-invariant measure $\widehat\mu_1=\mu_\psi$ which maximizes the free energy corresponding to $\psi$ (see \cite{Bowen}). When taking $\psi$ to be the geometric potential we get that $\mu$ is the unique SRB measure for $F_1$. In this case the conditions of Theorem \ref{BE} hold true (see \cite{GH}).

The conditions of Theorem \ref{LLT1} hold true when, in addition to the above assumptions,  the function $\widehat{A}$ cannot be written in the form 
$$
\widehat A=c+H-H\circ F_2+hZ
$$
for constants $c,h>0$, a function $H$ with $\|H\|<\infty$ and an integer valued function $Z$ (here $\|\cdot\|$ is the correspondin H\"older norm). We refer to \cite{GH,HH}. Such assumptions on $\widehat A$ are often refereed to as ``non-aperiodicity" or ``non-arithmeticity" of $\widehat A$. 
 The conditions of Theorem \ref{LLT2} (see again \cite{GH,HH}) are satisfied if instead of the above relation $\widehat A$ is integer valued and it cannot be written in the form 
$$
\widehat A=c+H-H\circ F_2+hZ
$$
for constants $c>0$ and $h>1$ a function $H$ with $\|H\|<\infty$ and an integer valued function $Z$. 
Such assumptions on $\widehat A$ are often refereed to as non irreducibility of an integer valued function $\widehat A$.

\section{Application to backward contracting on average maps $f_\omega$.}\label{App5}

\begin{theorem}\label{Thm}
Suppose $X$ is a compact metric space.
Let $T$ be a one-sided subshift of finite type and $\mu$ is a Gibbs measure associate with a H\"older continuous function $\varphi$ with exponent $\alpha$. Suppose that there exist $\epsilon,C,\lambda>0$  such that for every $x,y\in X$ with $0<d(x,y)\leq \epsilon$, we have
\begin{equation}\label{CAD0}
\int \frac{d^\alpha(f_\omega^{-N}x, f_\omega^{-N}y)}{d^\alpha(x,y)}d\mu(\omega)<1/C^2    
\end{equation}
where $C$ is the constant from the Gibbs property.
Assume also that  that $f_\omega$ are invertible and depend only on the zero-th coordinate. Let $\widehat\mu$ be an $F_1$-invariant probability measure on $\Sigma\times X$.
Then the conclusion of Theorem \ref{DEC} holds.

Moreover, the conditions of Theorems \ref{BE} and \ref{LLT1} 
holds true if $\widehat A$ is aperiodic in the sense that if cannot be written in the form 
$$
\widehat A\circ\pi_2=c+H-H\circ F_2+hZ
$$
for constants $c,h>0$, a function $H$ with $\|H\|<\infty$ and an integer valued function $Z$. 

If $A$ is integer valued then the conditions of Theorems \ref{BE} and \ref{LLT2} are in force
if $\widehat A$ is irreducible in the sense that if cannot be written in the form 
$$
\widehat A\circ\pi_2=c+H-H\circ F_2+hZ
$$
for constants $c>0$, $h>1$ a function $H$ with $\|H\|<\infty$ and an integer valued function $Z$. 
\end{theorem}
\subsubsection*{Proof of Theorem \ref{Thm}}
 Let
$$
\mathcal L_{1,t}g(\omega,x)=\sum_{T \omega'=\omega}e^{\varphi(\omega')+itA(f_{\omega'}x)}g(\omega', f_{\omega'}^{-1}x)
$$
and
$$
\mathcal L_{2,t}g(\omega,x,y)=\sum_{T \omega'=\omega}e^{\varphi(\omega')+it\widehat{A}(f_{\omega'}^{-1},f_{\omega'}^{-1}y)}g(\omega', f_{\omega'}^{-1}x,f_{\omega'}^{-1}y).
$$
\begin{lemma}
$\mathcal L_{1,0}$ is the dual of the Koopman operator of $F_1$ and $\mathcal L_{2,0}$ is the dual of the Koopman operator of $F_2$. Namely, for all measurable bounded functions we have $G_i, H_i$ we have
\begin{equation}\label{Dual1}
\int (G_1\circ F_1)H_1d\wh{\mu}=\int \mathcal L_{1,0}(H_1)G_1d\wh{\mu} 
\end{equation}
and 
$$
\int (G_1\circ F_2)H_1d\wh{\mu}^{(2)}= \int\mathcal L_{1,0}(H_2)G_2d\wh{\mu}^{(2)}.
$$
\end{lemma}
\begin{proof}
 By replacing $f_\omega$ with $f_\omega\times f_\omega$ it is enough to prove \eqref{Dual1}. Let us take two bounded measurable functions $G_1,H_1$. Then 
 $$
\int \mathcal L_{1,0}(H_1)G_1\wh{\mu}_1=\int\left(\int\sum_{\beta_0\to\omega_0}e^{\varphi((\beta_0,\omega)}H_1((\beta_0,\omega),f_{\beta_0}^{-1}t)G_1(\omega,t)d\widehat{\mu}_\omega(t)\right)d\mu(\omega) 
 $$
 $$
=\int\sum_{T\omega'=\omega}e^{\varphi(\omega')}\left(\int G_1(\omega',f_{\omega'}^{-1}x)H(\omega,x)d\widehat{\mu}_\omega(x)\right)d\mu(\omega)
 $$
 $$
=\int\int \mathcal L_\varphi(\int G_1(\cdot,f_\cdot^{-1}x)H_1(T(\cdot),x)d\mu_{T(\cdot)}(x)) d\mu(\omega)
 $$
 $$
=\int\left(\int G_1(\omega,f_\omega^{-1}x)H_1(T\omega,x)d\mu_{T_\omega}(x)\right)d\mu(\omega)
 $$
 $$
=\int\left(\int G_1(\omega,xy)H_1(T\omega,f_\omega y)d\mu_{\omega}(y)\right)d\mu(\omega)
 $$
 $$
=\int H_1 (G\circ F_1)d\widehat\mu_1
 $$
 where in the penultimate equality we use that $(f_\omega)_*\widehat{\mu}_\omega=\mu_{T\omega}$. Here 
 $$
\mathcal L_\varphi h(\omega)=\sum_{\beta_0\to \omega_0}e^{\varphi((\beta_0,\omega))}h(\beta_0,\omega).
 $$
\end{proof}
The following results follows from the previous lemma by induction.
\begin{corollary}\label{CharF}
Denote by $\textbf{1}$ the functiont taking the constant value $1$, regadless of its domain.
For every real $t$ we have
$$
\wh{\mu}(e^{itS_n A})=\int \mathcal L_{1,t}^n\textbf{1}d\widehat\mu
$$
and 
$$
\wh{\mu}^{(2)}(e^{itS_n \wh{A}})=\int \mathcal L_{2,t}^n\textbf{1}d\widehat\mu^{(2)}.
$$
\end{corollary}

Next, let $\alpha$ be the H\"older exponent of $\varphi$ and let $v_\alpha$ denote the H\"older constant corresponding to this exponent (regardless on the space $\Sigma, \Sigma \times X$ or $\Sigma\times X\times X$).
%We assume here that $X$ is a compact metric space and that there exist $\epsilon,C,\lambda>0$  such that for every $x,y\in X$ with $0<d(x,y)<\epsilon$,
%\begin{equation}\label{CAD}
%\int d^\alpha(f_\omega^{-n}x, f_\omega^{-n}y)d\mu(\omega)\leq %Ce^{-\lambda n}d^\alpha(x,y).    
%\end{equation}
%Note that when $X$ is a manifold this condition holds when 
%$$
%\sup_{(x,\xi)\in TX}\|d_xf_\omega^n\xi
%$$
\begin{lemma}[Lasota-Yorke inequality]
Under \eqref{CAD0},
there exist $C,\lambda_0>0$ such that
every real $t$, for $j=1,2$, $n\in\mathbb N$ and a H\"older continuous function $G$ with exponent $\alpha$ on $\Sigma\times X^j$ we have
$$
v_\alpha(\mathcal L_{j,t}^nG)\leq C(|t|+1)(e^{-\lambda_0 n}v_\alpha(G)+\|G\|_\infty).
$$
In particular, all the operators $\mathcal L_{j,t}$ have spectral radius less or equal to $1$ and they are quasi compact.
\end{lemma}
\begin{proof}
First, using the Gibbs property and induction we get that there exist $\lambda,C>0$ such that
\begin{equation}\label{CAD}
\int d^\alpha(f_\omega^{-n}x, f_\omega^{-n}y)d\mu(\omega)\leq Ce^{-\lambda n}d^\alpha(x,y).    
\end{equation}

Next, by replacing $f_\omega$ with $f_\omega\times f_\omega$ it is enough to prove the lemma with $j=1$.
Let us take two point $(\omega,t)$ and $(\tilde\omega,s)$ such that $\omega_0=\tilde\omega_0$ and $d(x,y)<\epsilon$. Then 
$$
\left|\mathcal L_{1,t}^nG(\omega,x)-\mathcal L_{j,t}^nG(\tilde \omega,y)\right|\leq I_1+I_2+I_3+I_4
$$
where with $\Gamma=\Gamma_{\omega_0,n}=\{\beta=(\beta_k)_{k=0}^{n-1}, \beta_{n-1}\to \omega_0\}$
,
$$
I_1=\|G\|_\infty\sum_{\beta\in\Gamma}|e^{S_n\varphi((\beta,\omega))}-e^{S_n\varphi((\beta,\tilde \omega))}|, 
$$
$$
I_2=|t|\|G\|_\infty\sum_{\beta\in\Gamma}e^{S_n\varphi((\beta,\tilde \omega))}|S_n A((\beta,\tilde \omega),x)-S_n A((\beta,\tilde \omega),y)|,
$$
$$
I_3=\sum_{\beta\in\Gamma}e^{S_n\varphi((\beta,\tilde \omega))}|G((\beta,\omega),f_{\omega}^{-n}x)-G((\beta,\tilde\omega),f_{\omega}^{-n}x))
$$
and 
$$
I_4=\sum_{\beta\in\Gamma}e^{S_n\varphi((\beta,\tilde \omega))}|G((\beta,\tilde \omega),f_{\omega}^{-n}x)-G((\beta,\tilde\omega),f_{\omega}^{-n}y)).
$$
Using the mean value theorem we have 
$$
\sum_{\beta\in\Gamma}|e^{S_n\varphi((\beta,\omega))}-e^{S_n\varphi((\beta,\tilde \omega))}|
$$
$$
\leq\sum_{\beta\in\Gamma} (e^{S_n\varphi((\beta,\omega))}+e^{S_n\varphi((\beta,\tilde \omega))})\|S_n\varphi((\beta,\omega))-S_n\varphi((\beta,\tilde \omega))\|\leq C\mathcal \|\mathcal L_\varphi^n\textbf{1}\|_\infty=C
$$
for some constant $C>0$. Thus, $I_1\leq C\|G\|_\infty$. Next, 
notice that $f_{(\beta,\omega)}$ depends only on $\beta$. Let us denote it by $f_\beta^{-n}$, Thus
$$
|S_n A((\beta,\tilde \omega),x)-S_n A((\beta,\tilde \omega),y)|
$$
$$
\leq \sum_{j=0}^{n-1}|A(f_{\beta}^{-j}x)-A(f_{\beta}^{-j}y)|\leq C\sum_{j=0}^{n-1}d(f_{\beta}^{-j}x,f_{\beta}^{-j}y)^\alpha.
$$
Therefore, by the Gibbs property,
$$
|I_2|\leq C|t|\|G\|_\infty\sum_{\beta}e^{S_n\varphi((\beta,\tilde \omega))}\sum_{j=0}^{n-1}d^{\alpha}(f_{\beta}^{-j}x,f_{\beta}^{-j}y)^\alpha
$$
$$
\leq C'|t|\|G\|_\infty\sum_{j=0}^{n-1}\int d^\alpha(f_{\beta}^{-j}x,f_{\beta}^{-j}y)d\mu(\beta)\leq C'|t|\|G\|_\infty.
$$
Now we estimate $I_3$. Note that
$$
|G((\beta,\omega),f_{\omega}^{-n}x)-G((\beta,\tilde\omega),f_{\omega}^{-n}x))\leq 2^{-n}v_\alpha(G)
$$
and so 
$$
I_3\leq v_\alpha(G)2^{-n}\sum_{\beta\in\Gamma}S_n\varphi((\beta,\tilde \omega))=v_\alpha(G)2^{-n}d^\alpha(x,y).
$$
Finally, we estimate $I_4$. We have
$$
I_4\leq v_\alpha(G) \sum_{\beta\in\Gamma}e^{S_n\varphi((\beta,\tilde \omega))}d^\alpha(f_\beta^{-n}x,f_\beta^{-n}y)
$$
and so by the Gibbs property,
$$
I_4\leq Cv_\alpha(G)\int d^\alpha(f_\beta^{-n}x,f_\beta^{-n}y)d\mu(\beta)\leq Cv_\alpha(G)e^{-\lambda n}d^\alpha(x,y).
$$
\end{proof}

Finally, we prove that $\mathcal L_{j,0}$ have a spectral gap with multiplictivity one (and thus also $\mathcal L_{j,t}$ for $|t|$ small). For this purpose we need to show that the covering condition of Buzzi \cite{Buzzi} holds. This is the content of the following lemma. To see why this lemma is enough to get the spectral gap we refer the readers to the arguments in the proof of \cite[Theorem 2.4]{DolgHafBE} in the case when all the maps $T_j$ from the coincide.
\begin{lemma}[Buzzi's covering condition]
Under \eqref{CAD0}, for every $a>0$ there exist $C_a, n(a)>0$ such that for all $n\geq n(a)$ and all functions $G\geq 0$ with $v_\alpha(G)\leq a$ and $\widehat{\mu}(G)=1$ and $H$ such that
$v_\alpha(H)\leq a$ and $\widehat{\mu}^{(2)}(H)=1$
we have 
$$
\inf_{(\omega,x)}\sum_{\beta\in\Gamma_{\omega_0,n}}e^{S_n\varphi((\beta,\omega))}G((\beta,\omega),f_{\beta}^{-n}x)\geq C_a
$$  
and 
$$
\inf_{(\omega,x,y)}\sum_{\beta\in\Gamma_{\omega_0,n}}e^{S_n\varphi((\beta,\omega))}H((\beta,\omega),f_{\beta}^{-n}x,f_{\beta}^{-n}y)\geq C_a,
$$
where  $\Gamma_{\omega_0,n}=\{\beta=(\beta_k)_{k=0}^{n-1}, \beta_{n-1}\to \omega_0\}$.
As a consequence, there are constants $C>0$ and $\delta\in(0,1)$ such that for every H\"older continuous function $G$ on $\Sigma\times X^j$, $j=1,2$, 
$$
\left\|\mathcal L_{j,0}^nG-\wh{\mu}^{(j)}(G)\textbf{1}\right\|_{\alpha}\leq C\delta^n\|G\|_\alpha
$$
where $\wh{\mu}^{(1)}:=\wh{\mu}$. 
Therefore, we have the exponential decay of correlations for the two point motion, and so the quenched effective exponential decay of correlations as in Section \ref{Exp} is in force.
\end{lemma}
\begin{proof}
To simplify the notation let us work with $\widehat{\mu}$. Otherwise, one can replace $f_\omega$ by $f_\omega\times f_\omega$. 
Let us fix such a function $G$ and take some $(\omega,x)\in\Sigma\times X$. Let $M$ be the constant from the topological mixing conditions of the subshift and let $n>M$. As in the previous proof, denote $\Gamma_{\omega_0,n}=\{\beta=(\beta_k)_{k=0}^{n-1}, \beta_{n-1}\to \omega_0\}$. 
For $\beta\in\Gamma_{\omega_0,n}$ write $\beta=(\gamma,c_\gamma)$ with $c_\gamma$ an admissible path of length $M$ from $\gamma_{n-M}$ to $\omega_0$. Then 
 $$
\sum_{\beta\in\Gamma_{\omega_0,n}}e^{S_n\varphi((\beta,\omega))}G((\beta,\omega),f_{\beta}^{-n}x)\geq \sum_{\gamma}e^{S_n\varphi((\gamma,c_\gamma,\omega))}G((\gamma,c_\gamma,\omega),f_{(\gamma,c_\gamma)}^{-n}x). 
 $$
 Now, because the H\"older constant of $G$ does not exceed $a$
 and using the Gibbs property we see that for every $\tilde c$ we have 
 $$
\sum_{\gamma}e^{S_n\varphi((\gamma,c_\gamma,\omega))}G((\gamma,c_\gamma,\omega),f_{(\gamma,c_\gamma)}^{-n}x)
$$
$$
\geq C_M\sum_{\gamma}e^{S_n\varphi((\gamma,\tilde c,\omega))}G((\gamma,\tilde c,\omega),f_{(\gamma,\tilde c)}^{-n}x)-Ce^{-\lambda(n-M)}.
$$
Indeed, we have 
$$
\sum_{\gamma}e^{S_n\varphi((\gamma,\tilde c,\omega))}|G((\gamma,c_\gamma,\omega),f_{(\gamma,c_\gamma)}^{-n}x)-G((\gamma,\tilde c,\omega),f_{(\gamma,\tilde c)}^{-n}x)|
$$
$$
\leq C\sum_{\gamma}e^{S_{n-M}\varphi((\gamma,\tilde c,\omega))}d^\alpha(f_\gamma^{(-n-M)}(f_{c_\gamma}^{-M}x),f_\gamma^{(-n-M)}(f_{\tilde c}^{-M}x))
$$
$$
\leq \sum_{c\in\mathcal A^M}
\sum_{\gamma}e^{S_{n-M}\varphi((\gamma,\tilde c,\omega))}d^\alpha(f_\gamma^{(-n-M)}(f_{c}^{-M}x),f_\gamma^{(-n-M)}(f_{\tilde c}^{-M}x))
$$
$$
\leq C\sum_{c\in\mathcal A^M}\int d^\alpha(f_\gamma^{(-n-M)}(f_{c}^{-M}x),f_\gamma^{(-n-M)}(f_{\tilde c}^{-M}x))d\mu(\gamma)
$$
$$
\leq C_Me^{-\lambda(n-M)}.
$$

Next, 
 by the Gibbs property, 
$$
\sum_{\gamma}e^{S_n\varphi((\gamma,\tilde c,\omega))}G((\gamma,\tilde c,\omega),f_{(\gamma,\tilde c)}^{-n}x)
\geq C\int G((\gamma,\tilde c,\omega),f_{(\gamma,\tilde c)}^{-n}x)d\mu(\gamma).
$$
Integrating with respect to $\tilde c$ (i.e. with respect to  $d\mu(\tilde c)$) and using the Gibbs property we conclude that 
$$
\sum_{\gamma}e^{S_n\varphi((\gamma,c_\gamma,\omega))}G((\gamma,c_\gamma,\omega),f_{(\gamma,c_\gamma)}^{-n}x)
\geq C\int G((\beta,\omega),f_{\beta}^{-n}x)d\mu(\beta)-Ce^{-\lambda(n-M)}.
$$

Now, let us fix some $\tilde \omega$ and $y$. Then, 
$$
G((\beta,\omega),f_{\beta}^{-n}x)\geq G((\beta,\tilde\omega),f_{\beta}^{-n}y)-a(2^{-n\alpha}+d^\alpha(f_{\beta}^{-n}x,f_{\beta}^{-n}y).
$$
Therefore, by the contraction on average condition \eqref{CAD},
$$
\int G((\beta,\omega),f_{\beta}^{-n}x)d\mu(\beta)\geq 
\int G((\beta,\tilde \omega),f_{\beta}^{-n}y)d\mu(\beta)-C'\delta^n 
$$
where $\delta\in(0,1)$ and $C'$ are constants. Let us take $(\tilde\omega,y)$ such that 
$$
\int G((\beta,\tilde \omega),f_{\beta}^{-n}y)d\mu(\beta)\geq \frac12\sup_{\omega',z}\int G((\beta,\omega'),f_{\beta}^{-n}z)d\mu(\beta).
$$
Therefore, for every $\omega'\in \Sigma$ and $z\in X$,
$$
\int G((\beta,\tilde \omega),f_{\beta}^{-n}y)d\mu(\beta)\geq \frac12\int G((\beta,\omega'),f_{\beta}^{-n}z)d\mu(\beta).
$$
Then, by integrating with respect to $\wh{\mu}_{T^n\omega'}$ and then with respect to $\mu$ we see that
$$
\int G((\beta,\tilde \omega),f_{\beta}^{-n}y)d\mu(\beta)\geq \frac12\int\int\int G((\beta,\omega'),f_{\beta}^{-n}z)d\wh{\mu}_{T^n\omega'}(z)d\mu(\beta)d\mu(\omega'). 
$$
Finally, notice that by the Gibbs property there is a constant $C''>0$ such that
$$
\int\int\int G((\beta,\omega'),f_{\beta}^{-n}z)d\wh{\mu}_{T^n\omega'}(z)d\mu(\beta)d\mu(\omega')
$$
$$
\geq C''\int\left(\int G(\zeta,f_{\zeta}^{-n}z)d\mu_{T^n\zeta}(z)\right)d\mu(\zeta)
$$
$$
=C''\int G(\zeta,z)d\widehat{\mu}(\zeta,z) =C''.
$$
\end{proof}

\begin{proof}[Proof of Theorem \ref{Thm}]
First, notice that the operators $\mathcal L_{j,t}$ are of class $C^\infty$ in $t$. Therefore, by using the general theory of quasi compact operators \cite{HH} together with Corollary \ref{CharF} we see that  the condition of Theorem \ref{BE} are in force if $A$ is not a coboundary with respect to $F_1$. Moreover, the conditions of Theorem \ref{LLT1} are in force if $\widehat A$ is aperiodic. Finally, the he conditions of Theorem \ref{LLT2} are in force if $A$ is integer valued but irreducible. 
\end{proof}

\section{Application to Markov chains in random Markovian environment}\label{App6}
Let $Y=(Y_n)$ be a stationary Markov chain, and suppose that $(\Sigma,\mathcal F,\mu,T)$ is the left shift system associated with $Y$. Let $\mathcal X$  a measurable space. Let $R_\omega(x,dy), x\in\mathcal X$ be random transition operators and suppose that $(X_{\omega,k})$ is a Markov chain with $R_{T^k\omega}$ being the transition density between step $k$ and step $k+1$. 
Let us denote the law of $X_\omega$ by $\kappa_\omega$. Let 
$$
\kappa=\int\kappa_\omega d\mu(\omega).
$$
Let $\pi_n:\Sigma\times\mathcal X^\mathbb Z$ be given by 
$$
\pi_n(\omega,z)=(\omega_n,z_n).
$$
We view $\pi_n$ as random variables with respect to the measure $\kappa$.
\begin{lemma}
Suppose that $R_{\omega}$ depends only on $\omega_0$. Then
the sequence $(\pi_n)$ is a homogeneous Markov chain and 
$$
\mathbb E[G(\pi_n)|\pi_{n-1}=(a,b)]=\mathbb E[R_{a}G(Y_1,\cdot)(b)|Y_0=a].
$$
Suppose also that the following Doeblin conditions hold: there exist probability measures $m_1$ and $m_2$ and constants $n_0\in\mathbb N$ and $\gamma>0$ such that for $\mu$-almost all $\omega\in\Sigma$, for all nonnegative bounded measurable functions $g$ and $h$,
$$
R_{\omega_{n_0-1}}\circ\cdots \circ R_{\omega_0}(g)\geq \gamma m_1(g)
$$
and with $Q$ denoting the Markov operator of $(Y_n)$, 
$$
Q^{n_0}(h)\geq \gamma m_2(h).
$$
Denote by $K$ the Markov operator of $(\pi_n)$. Then for all nonnegative bounded measureable functions $G$,
$$
K^{n_0}G\geq \gamma^2(m_1\times m_2)(G).
$$
In particular there is a unique  $K$-stationary measure $\nu$ and
\begin{equation}\label{RPF}
\sup_{\|G\|_{\infty}\leq1}\|K^nG-\nu(G)\|_\infty\leq C\delta^n    
\end{equation}
for some constants $C>0$ and $\delta\in(0,1)$, where $\|\cdot\|_\infty$ is the supremum norm. 
\end{lemma}
\begin{proof}
We have 
$$
\mathbb E_\kappa[G(\pi_n)H((\pi_k)_{k<n})]=\int\kappa_\omega (G(\omega_n,X_{\omega,n})H((\omega_k,X_{\omega,k})_{k<n}))d\mu(\omega)
$$
$$
=\int\kappa_\omega (R_{\omega_{n-1}}G(\omega_n,\cdot)(X_{\omega,n-1})H((\omega_k,X_{\omega,k})_{k<n}))d\mu(\omega)
$$
$$
=\int \left(R_{\omega_{n-1}}G(\omega_{n},\cdot)(x_{n-1})\right)H((\omega_k,x_k)_{k<n})d\kappa(\omega,x).
$$
Now the result follows by conditioning on $(\omega_k,x_k)_{k<n}$.

To prove the result about the Doeblin condition, note that
$$
K^nG(a,b)=\int H(y_1,...,y_n)Q(a,dy_1)Q(y_1,dy_2)\ldots Q(y_{n-1},dy_n)   
$$
where 
$$
 H(y_1,...,y_n)=\int G(y_n,z_n)R_{a}(b,dz_1)R_{y_1}(z_1,dz_2)\ldots R_{y_{n-1}}(z_{n-1},dz_n).
$$
Now, if $G$ is non-negative then under the Doeblin condition for $R$ we have
$$
 H(y_1,...,y_{n_0})\geq\gamma m_1(G(y_{n_0},\cdot)):=h(y_{n_0})
$$
and so under the Doeblin condition for $Q$,
$$
K^{n_0}G(a,b)\geq \gamma K^{n_0} h(a,b)\geq \gamma^2m_2(h)=\gamma^2\int G d(m_1\times m_2).
$$
\end{proof}
Note that \eqref{RPF} is sufficient to verify the conditions of Theorem \ref{BE} with bounded functions $A$ (see \cite{HH}). Moreover, by applying the results in \cite{HH}, the inequality \eqref{RPF} also yields that the non-lattice local CLT holds when $\widehat{A}$ is aperiodic (in the sense of \cite{HH}) and for integer valued functions $A$ the lattice local CLT holds when $\widehat{A}$ is not not reducible to an $h\mathbb Z$-valued function with $h>\pi$.

\begin{remark}
Although we do not need that in the paper, let us note that 
if also the upper Doeblin condition holds:
$$
R_{\omega_{n_0-1}}\circ\cdots \circ R_{\omega_0}(\cdot)\leq C m_1(\cdot)
$$
and
$$
Q^{n_0}(\cdot)\leq C m_2(\cdot)
$$
for some $C>0$ then 
$$
K^{n_0}\leq C^2(m_1\times m_2).
$$
Thus, also $K$ satisfies the two sided Doeblin condition.    
\end{remark}

\bibliographystyle{alpha}
\tocless\bibliography{Elphi}

\Addresses

\end{document}